\documentclass[USenglish]{article}
\usepackage[utf8]{inputenc}

\usepackage{amssymb,amsmath,amsthm,amscd,epsf,latexsym,verbatim,graphicx,amsfonts,epstopdf,xcolor,wasysym,color}
\input epsf.tex
\usepackage{bm}
\usepackage[utf8]{inputenc}
\usepackage[T1]{fontenc}
\usepackage{graphicx, enumerate}
\usepackage{palatino, url, multicol}
\usepackage{subfig}
\usepackage{geometry}                
\usepackage[linktoc=all]{hyperref}
\usepackage[all]{xy}
\usepackage{tikz}
\usetikzlibrary{shapes,snakes}

\theoremstyle{theorem}
\newtheorem{theorem}{Theorem}

\newtheorem{lemma}[theorem]{Lemma}
\newtheorem{proposition}[theorem]{Proposition}
\newtheorem{corollary}[theorem]{Corollary}
\newtheorem{claim}[theorem]{Claim}

\theoremstyle{definition}

\newtheorem*{remark}{Remark}

\newtheorem{definition}[theorem]{Definition}

\theoremstyle{question}

\newtheorem*{bouncetheorem}{Bounce Theorem}
\newtheorem*{supporttheorem}{Support Rigidity Theorem}

\def\Flat{\operatorname{Flat}}

\def\cyl{\operatorname{cyl}}
\def\SL{\operatorname{SL}}
\def\Chain{\operatorname{Chain}}

\def\B{{\sf B}}
\def\Q{{\mathbb Q}}

\def\R{{\mathbb R}}
\def\Z{{\mathbb Z}}

\def\Curr{\operatorname{C}}
\def\supp{\operatorname{supp}}
\def\G{{\mathcal G}}

\title{You can hear the shape of a billiard table: \\ Symbolic dynamics and rigidity for flat surfaces}
\author{Moon Duchin, Viveka Erlandsson, Christopher J. Leininger,\\ and Chandrika Sadanand%
\thanks{This work was initiated with funding from MD's NSF CAREER award DMS-1255442.  VE was partially supported by Academy of Finland project \#297258. 
CJL is partially supported by NSF grant DMS-1510034 and also acknowledges support from NSF grants DMS 1107452, 1107263, 1107367 "RNMS: GEometric structures And Representation varieties" (the GEAR Network).
CS acknowledges support from a Postdoctoral Fellowship at the Einstein Institute of Mathematics, Hebrew University.}}

\begin{document}
\maketitle



\begin{abstract}{We give a complete characterization of the relationship between the shape of a Euclidean polygon and the symbolic dynamics
of its billiard flow.  We prove that the only pairs of tables that can have the same bounce spectrum are right-angled tables that differ
by an affine map.  The main tool is a new theorem that establishes that a flat cone metric is completely determined 
by the {\em support} of its Liouville current.}
\end{abstract}



\setcounter{tocdepth}{2}
\tableofcontents


\section{Introduction}

There is a well-established line of  {\em spectral rigidity} problems in geometry, 
where one tries to show that various discrete invariants of an object determine its geometry. 
In 1966 Kac memorably asked if one could {\em hear the shape of a drum}; or more precisely whether the shape of a planar domain is determined by the spectrum of its Laplacian \cite{Kac}.
In this paper we define a ``bounce spectrum" for polygons, recording the symbolic dynamics
of the billiard flow, and we explore the question of  how precisely this determines the polygon, and thus whether one can "hear" the shape of the polygon from the bounce spectrum.

Every Euclidean polygon $P$ supports a {\em billiard flow}, in which a point mass in the interior travels in 
a straight line until it hits a wall, then bounces off the wall with {\em optical reflection}:  angle of incidence equals angle of reflection.
Trajectories hitting corners are regarded as singular and discarded.  A polygon together with this billiard flow constitutes a dynamical system
called a billiard table (or simply {\em table}) below.

Suppose the sides of a simply connected (finite-sided, not necessarily convex) 
polygon $P$ are labeled in cyclic order by letters from an ordered alphabet $\mathcal A$.
Then for a nonsingular billiard path,  $\gamma \colon \mathbb R \to P$, we can record 
the labels of the sides of $P$ in the order that they are encountered along $\gamma$, 
forming a bi-infinite sequence of labels called the {\em bounce sequence}, an element of $\mathcal A^\Z$.  
Here the zeroth term of the sequence corresponds to the first side encountered on $\gamma([0,\infty))$, and so starting sooner or later along a trajectory amounts to shifting the corresponding
bounce sequence.
The set of all bounce sequences that can occur on $P$ is denoted $\B(P)$ and called the {\em bounce spectrum}
of the billiard table.  
This closely resembles the construction of {\em cutting sequences}, a fundamental concept in symbolic 
dynamics.
If there exist cyclic labelings of the sides of polygons $P_1$ and $P_2$ with letters from $\mathcal A$ which induce a bijection of bounce spectra, then
 we will identify the spectra and write $\B(P_1)=\B(P_2)$. 

It is immediately clear that applying a {\em similarity} (any combination of dilation, reflection, and rotation)  to a polygonal table
does not change its bounce spectrum.
Upon investigation, one quickly observes that the coordinate-affine map given by $\left(\begin{smallmatrix} a&0 \\ 0 & b\end{smallmatrix}\right)$ preserves 
optical reflection in horizontal and vertical edges.  
It follows that any two rectangles $R_1,R_2$  have the same bounce spectrum $\B(R_1)=\B(R_2)$, and more generally that
 $\B(P_1)=\B(P_2)$ for any two right-angled tables (all angles in $\frac\pi 2 \mathbb N$) 
that are related by an affine map.   
Outside of this exception, we get the strongest possible rigidity result:  the bounce spectrum determines not only the precise angles 
in the polygon, but also the proportions of sidelengths.

\begin{bouncetheorem} If two simply connected Euclidean polygons $P_1,P_2$ have 
$\B(P_1)=\B(P_2)$, then either $P_1,P_2$ are right-angled and affinely equivalent, or they are 
similar polygons.
\end{bouncetheorem}

\begin{remark} Pursuing the analogy between rigidity for the bounce spectrum and the length spectrum, we can regard
the cyclic labeling of edges as a way to identify sides in one polygon with sides in the other; this is like the {\em marked} length 
spectral problem for manifolds, in which a correspondence of curves is specified.
 It is natural to wonder about the unmarked problem for billiard tables.
In a new preprint, Calderon-Coles-Davis-Lanier-Oliveira \cite{constructive} prove that for an arbitrary labeling of the sides of a polygon,  
one can "hear" edge adjacency (i.e., detect which labels correspond to incident edges) from the  bounce sequences.
Combining this result with ours, we obtain bounce-spectral rigidity in the unmarked case as well.

This theorem also shows that the Laplace spectrum and the bounce spectrum contain different information.  It is well known that the Laplace spectrum determines the area and perimeter of 
a plane domain, from which it follows that rectangles can be distinguished by their Laplace spectra, while all having the same bounce spectrum.  
On the other hand, the "propeller pair" constructions of  Gordon--Webb--Wolpert \cite{GWW} have the same Laplace spectrum but, by our main theorem, different bounce 
spectra.\footnote{See also \cite{constructive,LuRowlett}, in which the authors investigate what geometric properties of a polygon can be reconstructed from just a finite part 
of the bounce spectrum, Laplace spectrum, or length spectrum.}
\end{remark}

In the special case of rational tables (all angles in $\pi\Q$), there is a well-established toolkit for studying the billiard flow
centered on the "unfolding" 
to a closed {\em translation surface} and all the machinery that comes with that.
A recent sequence of papers of Bobok--Troubetzkoy \cite{BT1,BT2,BT3} 
culminates in a proof of symbolic rigidity (as in the Bounce Theorem above) when one of the tables is assumed to be rational.
In fact, their results are stronger in that setting: they show that only part of the bounce spectrum 
(the bounce sequence of a single generic trajectory, or the subset of periodic sequences) suffices to determine the table.

Our proof works for general tables, and we note that this  setting is quite different; for instance,
rational tables have periodic trajectories in a dense set of directions, while it is unknown if all irrational tables (or even all triangular ones!)
have even a single periodic trajectory.\footnote{Indeed, in \cite{SchwartzObI}, Schwartz proves that there is a sequence of triangles converging to the 
$(30,60,90)$ triangle for which the combinatorial length of the shortest periodic billiard trajectory tends to infinity.  See also Hooper \cite{Hooper} for more on instability of periodic trajectories.}

Our proof utilizes essentially the entire bounce spectrum in a crucial way.  On the other hand, any set that determines the bounce spectrum clearly also determines the shape of the table by the  Bounce Theorem.  For example, we have the following corollary.  A {\em generalized diagonal} in a Euclidean polygon $P$ is a billiard trajectory $\gamma \colon [a,b] \to P$ that starts and ends at a vertex of the polygon.  Given a cyclic labeling of the sides of $P$, a generalized diagonal induces a {\em finite} bounce sequence, and we let $\B_\Delta(P)$ denote the {\em countable} set of bounce sequences of generalized diagonals.

\begin{corollary} \label{C:gen diag} If two simply connected Euclidean polygons $P_1,P_2$ have $\B_\Delta(P_1)=\B_\Delta(P_2)$, then either $P_1,P_2$ are right-angled and affinely equivalent, or they are 
similar polygons.
\end{corollary}

The Bounce Theorem is a consequence of a result about geodesic currents associated to flat metrics.
To state this result, we require a few definitions (see \S\ref{S:preliminaries} below for fuller background). 
The space of equivalence classes of {\em flat metrics} (nonpositively curved Euclidean cone metrics)
on a closed, oriented surface $S$  is denoted $\Flat(S)$.  Associated to each metric $\varphi \in \Flat(S)$ is a geodesic current $L_\varphi$ called the {\em Liouville current}, 
which is formally a measure on the double boundary of the universal cover of $S$.  
We affirmatively answer an open question of Bankovic-Leininger showing that the support of $L_\varphi$ determines $\varphi$, up to affine deformation (see \cite[Section 6]{BL}). That is, a flat metric is not only determined by its geodesic current but even by the support alone, which is quite different from the hyperbolic case where currents have full support. This Support Rigidity Theorem yields the Bounce Theorem via a new unfolding technique we introduce below.
\begin{supporttheorem}
Suppose $\varphi_1,\varphi_2$ are two unit-area flat metrics whose Liouville currents have the same support,
$\supp(L_{\varphi_1}) = \supp(L_{\varphi_2})$.  Then $\varphi_1,\varphi_2$ differ by an affine deformation, up to isotopy.  If either metric has holonomy of order greater than $2$ (i.e., is not induced by a quadratic differential, or does not support a foliation by straight lines), 
then equal support implies that $\varphi_1$ and $\varphi_2$ differ by isometry, isotopic to the identity.
\end{supporttheorem}

We briefly describe the steps involved in the proof of the Support Rigidity Theorem.
A key observation for carrying out this program is that the support of $L_\varphi$ consists precisely of the closure of the set of nonsingular $\tilde \varphi$--geodesics in $\tilde S$; see Proposition~\ref{P:support_are_basic_endpoints} in \S\ref{S:chains section}.  
We then use a technical device called chains (see \S 2.5) to fix an identification of the cone points, and from these an identification of the saddle connections and the nonsingular geodesics, between $\varphi_1$ and $\varphi_2$.
This lets us build a careful correspondence of triangulations between $\varphi_1$ and $\varphi_2$, and even of directions of travel
as trajectories cross the edges of a triangulation.
This can be used to show that the holonomy group (the rotations observed
when transporting a tangent vector around a loop) is the same for the two metrics.  

Our goal is to construct an affine map $(S,\varphi_1) \to (S,\varphi_2)$ isotopic to the identity.  We carry this out by picking a geometric triangulation of $S$ with respect to $\varphi_1$ and then proving that the map to the corresponding triangulation in $\varphi_2$
can be adjusted by an isotopy first on the vertices, then on the edges, and finally triangles, to produce a map which is affine on each of the triangles.  Next we must analyze 
the piecewise affine maps and show that they are globally affine, and in fact isometries, up to the controlled exceptions identified in the 
theorem statement.

To prove the Bounce Theorem from the Support Rigidity Theorem, we consider arbitrary nonpositively curved "unfoldings" $X$ of a polygon 
$P$  and observe that, as with the  unfoldings to translation surfaces in the rational case, the nonsingular geodesics on $X$ correspond to billiard trajectories in $P$.  
We have seen that the set $\supp(L_\varphi)$ coarsely encodes (the closure of) the set of nonsingular $\varphi$--geodesics.  
The bounce spectrum of $P$ records how geodesics cut through edges of a triangulation and  can thus be viewed as providing the same coarse information.  
From two polygons $P_1,P_2$ with $\B(P_1) = \B(P_2)$,  we can find a pair of nonpositively curved unfoldings $X_1,X_2$ with a common underlying topological surface $S$ such that their  nonsingular geodesics can be identified.  
Appealing to the Support Rigidity Theorem, we conclude that
$X_1,X_2$ differ by an affine map (generically an isometry, or similarity if we do not normalize the areas).  With some care, this  
induces a suitable affine map between $P_1$ and $P_2$. For Corollary~\ref{C:gen diag}, we observe that generalized diagonals correspond to saddle connections in the unfolding and use limits of codings to understand the nonsingular geodesic as limits of saddle connections.

The ideas used to deduce the Bounce Theorem from the Support Rigidity Theorem can be applied to other settings as well:  
any data that coarsely determines the nonsingular geodesics of a flat surface can be seen to determine the flat metric, up to affine equivalence.  For example, in Section~\ref{bounce} we sketch a similar result for {\em cutting sequences}.

\subsection{Connections to symbolic dynamics literature.}  The results of Bobok and Troubetzkoy mentioned above 
\cite{BT1,BT2,BT3} are directly relevant to our work, 
and the results of Calderon et al \cite{constructive} give constructive results on the exact questions treated here.
These are part of a substantial literature 
relating  symbolic dynamics  to the geometry of billiards and associated flat surfaces.  Variants on bounce sequences 
include cutting sequences and Sturmian sequences.
Characterizing the sequences occurring in Veech surfaces (such as regular polygons and square-tiled surfaces) 
appears in the work of Morse and Hedlund from the 1930s and 40s \cite{MorseHedlund,MorseHedlundII} and includes very recent work of 
Smillie--Ulcigrai \cite{SmillieUlcigrai,SmillieUlcigrai2}, Davis \cite{Davis1,Davis2},  Davis--Pasquinelli--Ulcigrai \cite{DavPasUlc}, and 
Johnson \cite{CCJohnson}. Beyond the Veech case, there is a long string of papers of Lopez--Narbel developing a language-theoretic
formulation of generalized Sturmian sequences for interval exchange transformations, beginning with \cite{LopezNarbel}.

Complexity of billiards  has been studied via symbolic coding by Katok \cite{KatokGR}, Hubert \cite{HubertComp},  Troubetzkoy \cite{TroubetzkoyComp}, and  Hubert--Vuillon \cite{HubertVuillon}.  
From the point of view of determining geometric information from a bounce sequence, Galperin, Kr\"uger, and Troubetzkoy \cite{GKT}, for example, prove sharp relationships between periodic trajectories and periodic bounce sequences, which in turn influenced the Flat Strip Theorem of Hassell--Hillairet--Marzuola \cite{HHM} (Theorem~\ref{T:flat strips} below).  
We refer the reader to these
works and their references for further discussion of the many connections between symbolic coding and geometry.

\subsection*{Acknowledgments} The authors would like to thank all the participants and visitors to the Polygonal Billiards Research Cluster held at Tufts University in Summer 2017.
In particular, we thank Pat Hooper, Rich Schwartz,  Caglar Uyanik, and the authors of \cite{constructive} for illuminating conversations on billiards and symbolic dynamics.   Erlandsson and Leininger would like to thank the School of Mathematics at Fudan University and the Mathematics Research Centre and the University of Warwick, respectively, for their hospitality while this project was being completed.

\section{Preliminaries} \label{S:preliminaries}

In this section we discuss background, establish notation, and give preliminary results we will use in this paper.  
Throughout, $S$ will denote a closed, oriented surface of genus at least $2$ and $p \colon \tilde S \to S$ will denote the universal 
covering map.  We also fix an action of $\pi_1S$ on $\tilde S$ by covering transformations.

\subsection{Spaces of geodesics} \label{S:spaces of geodesics}

Fix once and for all an arbitrary hyperbolic metric $\rho$ on $S$, and let $\tilde \rho = p^*(\rho)$ be the pullback to $\tilde S$, thus 
specifying an identification of  $\tilde S$ with the hyperbolic plane.  Let $S^1_\infty$ denote the circle at infinity bounding $\tilde\rho$, and equip it with the action of $\pi_1S$ obtained by extending the action on $\tilde S$.  Given any other geodesic metric $m$ on $S$, with the pullback $\tilde m = p^*(m)$ on $\tilde S$, the identity on $\tilde S$ is a $\pi_1S$--equivariant quasi-isometry to $\tilde \rho$, and hence the Gromov boundary of $\tilde m$ is identified with $S^1_\infty$; see \cite[Chapter III.H.3]{BridHaef}.  In particular we view $S^1_\infty$ as the Gromov boundary of the pullback of {\em any} metric on $S$.

Let $\G(\tilde S)$ denote the set of unordered pairs of distinct points in $S^1_\infty$, that is, 
\[ \G(\tilde S) = \{ \{x,y\} \mid x,y \in S^1_\infty, \, x \neq y \} = \left(S^1_\infty \times S^1_\infty \setminus\Delta\right)/\sim,\]
where $\Delta$ is the diagonal and $\sim$ is the equivalence relation $(x,y)\sim(y,x)$.  The action of $\pi_1S$ on $S^1_\infty$ determines an action on $\G(\tilde S)$.

For any geodesic metric $m$ on $S$ an $\tilde m$--geodesic on $\tilde S$ is a map $\tilde \gamma \colon I \to \tilde S$ which is an isometric embedding from a (finite, infinite, or bi-infinite) interval  $I \subset \mathbb R$ to $(\tilde S,\tilde m)$.  
Below, we will use the term {\em geodesic} for the map and its image interchangeably.  If we need to distinguish them, we 
will refer to parametrized and unparametrized geodesics.  Downstairs on $S$, an $m$--geodesic 
$\gamma$ is the composition of an $\tilde m$--geodesic with the covering projection; that is, $\gamma = p \circ \tilde \gamma \colon I \to S$.  
These are locally isometric embeddings, but for arbitrary metrics, not all locally isometric embeddings to $S$ are obtained in this way.  However, 
for nonpositively curved metrics (e.g., the hyperbolic metric $\rho$ or flat metrics), every locally isometric embedding $\gamma \colon I \to S$ lifts 
to an isometric embedding since $\tilde S$ is a CAT(0) space;  see \cite[Proposition~II.1.4]{BridHaef}.  

We let $\G(\tilde m)$ denote the space of all unparametrized bi-infinite $\tilde m$--geodesics on $\tilde S$ given the Chabauty-Fell topology (see \cite{Fell,Chabauty} and \cite{CEG}) equipped with the action of $\pi_1S$ induced by its action on $\tilde S$.  A sequence $\{\tilde \gamma_n\}_{n=1}^\infty$ in $\G(\tilde m)$ converges to a geodesic $\tilde \gamma$ if and only if the geodesics can be parametrized so that the sequence of parametrized geodesics converges uniformly on compact sets to the parametrization of $\tilde \gamma$.\footnote{In \cite{BL}, the topology of 
$\G(\tilde m)$ is incorrectly stated to be the quotient of the compact-open topology by forgetting the parametrizations.  
The property that convergence implies the existence of parametrizations which converge locally uniformly is all that is used, however.}
There is a continuous, closed, proper, surjective map 
\[ \partial_{\tilde m} \colon \G(\tilde m) \to \G(\tilde S)\]
defined by setting $\partial_{\tilde m}(\tilde \gamma)$ to be the endpoints at infinity of the geodesic.  When $\tilde m$ has negative curvature, 
$\partial_{\tilde m}$ is a homeomorphism---in particular, $\partial_{\tilde \rho}$ is a homeomorphism.
Below, we will build up to a better understanding of $\partial_{\tilde \varphi}$ for the case of flat metrics $\varphi$.

We say that pairs $\{x,y\},\{x',y' \} \in \G(\tilde S)$ {\em link} if the two points $x'$ and $y'$ are in different components of $S^1_\infty \setminus \{x,y\}$.  A pair of $\tilde m$--geodesics with linking endpoints necessarily intersect each other, though even when $\tilde m$ is CAT(0),
the intersection may not be transverse (see the left-hand side of Figure~\ref{F:crossing and betweenness}).  
If $\{x,y\},\{x',y' \} \in \G(\tilde S)$ do not link, then we say that $\{x'',y''\}$ is {\em between} $\{x,y\}$ and $\{x',y'\}$ if, with respect to some choice of ordering of each of the three pairs, the six points appear cyclically as $x \leq x'' \leq x' \leq y' \leq y'' \leq y$ as in the right-hand side of Figure~\ref{F:crossing and betweenness}.  The terminology is suggestive of the behavior of geodesics with those endpoints (although for general metrics, the actual intersection patterns of geodesics can be more complicated).  We write $[\{x,y\},\{x',y'\}]$ for the set of all pairs $\{x'',y''\}$ between $\{x,y\}$ and $\{x',y'\}$.

\begin{figure}[htb]
\begin{center}
\begin{tikzpicture}[scale = .8]
\draw (0,0) circle (3);
\draw (2.12132,-2.12132) -- (0,-1) -- (0,1.5) -- (.7764,2.89777);
\draw (-.7764,-2.89777) -- (0,-1) -- (0,1.5) -- (-2.12132,2.12132); 
\draw[fill=black] (2.12132,-2.12132) circle (.04cm);
\draw[fill=black]  (.7764,2.89777) circle (.04cm);
\draw[fill=black]  (-.7764,-2.89777) circle (.04cm);
\draw[fill=black]  (-2.12132,2.12132) circle (.04cm);
\node at (-.8,-2) {$\tilde \gamma$};
\node at (-1.5,1.6) {$\tilde \gamma'$};
\node at (.7,2.2) {$\tilde \gamma$};
\node at (1.6,-1.5) {$\tilde \gamma'$};
\node[above] at (.7764,2.89777) {$x$};
\node[below] at (-.7764,-2.89777) {$y$};
\node[above] at (-2.12132,2.12132) {$x'$};
\node[below] at (2.22132,-2.02132) {$y'$};
\draw (8,0) circle (3);
\draw [domain=-45:45] plot ({3.75736+3*cos(\x)}, {3*sin(\x)});
\draw [domain=135:225] plot ({12.24264+3*cos(\x)}, {3*sin(\x)});
\draw (8,3) -- (8,-3);
\draw[fill=black] (10.12132,-2.12132) circle (.04cm);
\draw[fill=black] (10.12132,2.12132) circle (.04cm);
\draw[fill=black] (5.87868,-2.12132) circle (.04cm);
\draw[fill=black] (5.87868,2.12132) circle (.04cm);
\draw[fill=black] (8,3) circle (.04cm);
\draw[fill=black] (8,-3) circle (.04cm);
\node[below] at (10.12132,-2.12132) {$x'$};
\node[above] at (10.12132,2.12132) {$y'$};
\node[below] at (5.87868,-2.12132) {$x$};
\node[above] at (5.87868,2.12132) {$y$};
\node[above] at (8,3) {$y''$};
\node[below] at (8,-3) {$x''$};
\node[right] at (8,.2) {$\tilde \gamma''$};
\node[left] at (6.5,-1) {$\tilde \gamma$};
\node[right] at (9.5,1) {$\tilde \gamma'$};
\end{tikzpicture}
\caption{{\bf Left:} Two geodesics with linking endpoints that do not meet transversely. {\bf Right:} A geodesic $\tilde \gamma''$ between geodesics $\tilde \gamma$ and $\tilde \gamma'$, illustrating the notion of betweenness of their endpoints, $\{x'',y''\} \in [\{x,y\},\{x',y'\}]$ in $\G(\tilde S)$.}
\label{F:crossing and betweenness}
\end{center}
\end{figure}
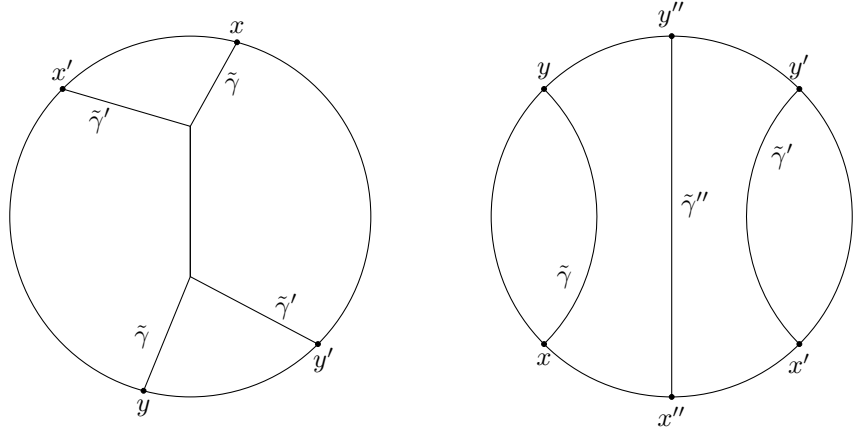

\subsection{Geodesic currents}\label{sec_currents}

A \emph{geodesic current} on $S$ is a $\pi_1S$-invariant Radon measure on $\mathcal{G}(\tilde{S})$.  Let $\Curr(S)$ denote the set of all geodesic currents on $S$ with the weak* topology.  Any essential closed curve $\gamma$ on $S$ determines a geodesic current of the same name which is the counting measure on the set of pairs of endpoints on $S^1_\infty$ of $p^{-1}(\gamma)$.  In fact, Bonahon proved that 
currents are a completion of the set of closed curves, in the sense that 
the set of real multiples of associated currents
$$\{t\cdot\gamma\,|\,t\in\mathbb{R}_+, \gamma\subset S\text{ closed curve}\}$$
forms a dense subset of $\Curr(S)$.  The geometric intersection number between two closed curves $\gamma$ and $\delta$ is the minimal number of transverse double-points of intersection among all curves $\gamma'$ and $\delta'$ homotopic to $\gamma$ and $\delta$, respectively.  This is realized  by the $\rho$--geodesic representatives, and Bonahon proved that the geometric intersection number has a continuous extension to the full space of  currents.
\begin{theorem}[\cite{Bonahon,Bonahon2}] \label{T:geometric intersection form}
There is a continuous, bilinear function
\[ \iota \colon \Curr(S) \times \Curr(S) \to \mathbb R \]
such that for every pair of closed curves $\gamma,\delta$, with associated geodesic currents of the same name, $\iota(\gamma,\delta)$ 
recovers the geometric intersection number.
\end{theorem}

Another important class of geodesic currents is the  \emph{Liouville currents} $L_m$ associated to certain types of metrics $m$.  
Liouville currents have the geometricity property that the intersection form recovers lengths: for every essential closed curve $\gamma$ on $S$,
\begin{equation}\label{intersection}
\iota(\gamma, L_m)=\ell_{m}(\gamma),
\end{equation}
where $\ell_m(\gamma)$ is the length of the $m$--geodesic representative of the homotopy class of $\gamma$.  Such currents exist for a wide range of metrics, generalizing the classical Liouville measure on geodesics in the hyperbolic plane given in terms of  cross-ratios; see \cite{Bonahon,Otal,croke,CFF,HerPaul,DLR,BL,Constantine}.  The construction of the Liouville current for a flat metric and an investigation of its various properties was the focus of \cite{DLR} and \cite{BL}.  We describe the key properties in \S\ref{S:flat liouville}, setting the stage for the more detailed analysis we will carry out in \S\ref{sec_support}. 

\subsection{Flat metrics, geodesics, and holonomy} \label{S:flat basics}\label{S:basic}

A {\em flat metric} on $S$ will mean a nonpositively curved Euclidean cone metric $\varphi$ on $S$.  This is a singular Riemannian metric, locally isometric to $\mathbb R^2$ away from finitely many cone singularities $\Sigma=\Sigma(\varphi) \subset S$, each with cone angle greater than $2 \pi$. 
Since $\varphi$ is nonpositively curved, any two points of $\tilde S$ are connected by a unique $\tilde \varphi$--geodesic.  In the complement of the cone points, geodesics are Euclidean geodesics (straight lines, rays, or segments), and when a geodesic meets a cone point, it makes angle at least $\pi$ on both sides; see Figure~\ref{F:local geometry of geodesics}.  

\begin{figure}[h]
\begin{center}
\begin{tikzpicture}[scale = .6]
\draw[dotted] (-2,-2) -- (-1.75,-1.75);
\draw (-1.75,-1.75) -- (0,0) -- (1,2);
\draw[dotted] (1,2) -- (1.12,2.24); 
\draw [domain=-135:60] plot ({.3*cos(\x)}, {.3*sin(\x)});
\draw [domain=60:225,ultra thick] plot ({.4*cos(\x)}, {.4*sin(\x)});

\draw[fill=black] (0,0) circle (.05cm);
\node at (.8,-.3) {\small $\geq \pi$};
\node at (-.5,.7) {\small $\geq \pi$};
\end{tikzpicture}
\caption{Local picture of a geodesic through a cone point.}
\label{F:local geometry of geodesics}
\end{center}
\end{figure}
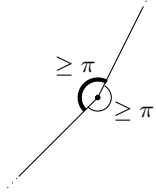

Geodesics between a pair of cone points which meet no other cone points are called {\em saddle connections}.  A geodesic segment, ray, or line containing no cone points is called {\em nonsingular}, and is called {\em singular} otherwise.   We 
write $\G^*(\tilde \varphi)$ for the closure in $\G(\tilde \varphi)$ of the set of nonsingular geodesics, and we will call these the 
{\em basic geodesics} in the metric $\tilde \varphi$.
The following is proved in \cite[Section 2.4]{BL}.

\begin{proposition}[Structure of basic geodesics]\label{P:support geodesic types}\label{P:basic}
The geodesics in $\G^*(\tilde \varphi)$ are precisely the following $\tilde \varphi$--geodesics:
(i) the nonsingular ones; (ii) the ones that meet a single cone point and make cone angle exactly $\pi$ on one side; and 
(iii) those that meet two or more cone points in such a way that they always make cone angle $\pi$ on one side and that side switches
from right to left or from left to right at most once along their entire length.  
Furthermore, there are only countably many of type (iii), and we will denote that set by $\G^2(\tilde \varphi)$.
\end{proposition}

Suppose $\tilde \gamma_0,\tilde \gamma_1 \in \G^*(\tilde \varphi)$ are asymptotic in one direction, meaning that the 
Hausdorff distance in one direction is finite.  Then either we can pass to subrays
$\tilde \gamma_0^+ = \tilde \gamma_1^+$ 
that coincide, 
or else there is an isometric embedding $[0,\infty) \times [0,W] \to \tilde S$ for some $W >0$ such that $[0,\infty) \times \{0\}$ maps to $\tilde \gamma_0^+$ and $[0,\infty) \times \{W \}$ maps to $\tilde \gamma_1^+$;  see \cite[Theorem II.2.13]{BridHaef}.  In the former case, there is a maximal such common subray, necessarily emanating from a cone point $\tilde \zeta \in \Sigma(\tilde \varphi)$, and we say that $\tilde \gamma_0$ and $\tilde \gamma_1$ are {\em cone-point asymptotic}.  In the latter case, the embedding of $[0,\infty) \times [0,W]$ is called a {\em flat half-strip}.

A {\em flat strip} is an isometric embedding $\mathbb R \times [0,W] \to \tilde S$, for some $W> 0$, or the composition of such an isometric embedding with the projection down to $S$.  Flat strips in $S$ naturally arise from Euclidean {\em cylinders} in $S$, which are locally isometric immersions $S^1_r \times [0,W] \to S$ where $W >0$ and $S^1_r$ is a circle of length $r >0$.  The next result of Hassell--Hillairet--Marzuola \cite{HHM} (generalizing work of Galperin-Kr\"uger-Troubetzkoy \cite{GKT} for billiards) says that all flat strips arise in this way.  
The proofs in \cite{GKT,HHM} are easily modified to handle flat half-strips as well.

\begin{theorem}[Flat Strip Theorem, \cite{GKT,HHM}]\label{T:flat strips} With respect to the metric $\varphi\in \Flat(S)$, any flat strip or flat half-strip in $S$ is contained in a cylinder.
\end{theorem}

A {\em $\tilde \varphi$--Euclidean triangle} is a $\tilde \varphi$--geodesic triangle in $\tilde S$ with vertices at the cone points which is isometric to a geodesic triangle in $\mathbb R^2$; equivalently, a $\widetilde \varphi$--Euclidean triangle is a triangle whose sides are $\widetilde \varphi$--saddle connections and which bounds a $2$--simplex with no cone points in its interior.  (Note: we may refer to either the $2$--simplex or its boundary as  
a triangle.)  A $\varphi$--Euclidean triangle $T$ is the image in $S$ of a $\widetilde \varphi$--Euclidean triangle $\tilde T$ in $\tilde S$ for which the restriction of $p \colon \widetilde S \to S$ to $\tilde T$ is injective on the interior.  In this case, there is a map from $T$ to a Euclidean triangle in $\mathbb R^2$ (though it may not be defined at the vertices) which is an isometry for the induced path metric on $T$.

A  $\varphi$--{\em triangulation} of $S$ is a triangulation (more precisely, a $\Delta$--complex structure in the sense of \cite{Hatcher}) such that every triangle is a $\varphi$--Euclidean triangle.  A $\varphi$--triangulation lifts to a $\tilde \varphi$--triangulation of $\tilde S$.  It was shown in \cite{MasurSmillie} that $\varphi$--triangulations exist for quadratic differential metrics, and we note that  the same proof is valid for arbitrary flat metrics.

For a flat metric $\varphi$, any homotopy class of a curve $\gamma$ on $S$  has a $\varphi$--geodesic representative; such a representative can be obtained as a uniform limit of constant-speed parametrizations of homotopic curves whose lengths limit to the infimum over all representatives.  The geodesic representative is either unique (up to parametrization) or else it is part of a  Euclidean {\em cylinder}, $S^1_r \times [0,W] \to S$, as above.
In this case the set of geodesic representatives of $\gamma$ is precisely the set of {\em core curves} (i.e., the images of $S^1_r \times \{t\}$ for $t \in [0,W]$).  If a curve $\gamma$ is homotopic to a core curve of  a cylinder, then we say that $\gamma$ is a cylinder curve.  The set of all homotopy classes of cylinder curves for $\varphi$ will be denoted $\cyl(\varphi)$.  In general, {\em closed curves} will  
refer to closed essential curves on $S$ up to homotopy.

Suppose $\varphi$ is a flat metric on $S$.  Since $\varphi$ has zero curvature on $S \setminus \Sigma(\varphi)$, we have a well-defined {\em holonomy homomorphism} $P_{\varphi} \colon \pi_1(S \setminus \Sigma(\varphi),\zeta) \to SO(2)$ given by parallel transport around loops based at a point $\zeta \in S \setminus \Sigma(\varphi)$.  
We will often refer to the image of $P_{\varphi}$ as {\em the holonomy of $\varphi$}.

\subsection{The space $\Flat(S)$ and affine equivalence} \label{S:Flat(S)}

We say that two flat metrics $\varphi_1,\varphi_2$ are equivalent if there exists an isometry $f \colon (S,\varphi_1) \to (S,\varphi_2)$ isotopic to the identity on $S$.  We then define $\Flat(S)$ to be the set of equivalence classes of unit-area flat metrics.  
We will use the notation $\varphi$ for either a particular flat metric or its equivalence class.

We will also need the following coarser equivalence relation on flat metrics.  
 We say that two (unit area) flat metrics $\varphi_1,\varphi_2$ are {\em affine--equivalent} if there is a map $f \colon (S,\varphi_1) \to (S,\varphi_2)$, isotopic to the identity, which is real-affine in isometric local coordinates: for all $\zeta \in S \setminus \Sigma(\varphi_1)$ and orientation preserving, isometric local coordinates $\xi_1 \colon U_1 \to \mathbb R^2$ and $\xi_2 \colon U_2 \to \mathbb R^2$ about $\zeta$ and $f(\zeta)$, respectively, with $f(U_1) = U_2$, we have $\xi_2 \circ f|_{U_1} = A \circ \xi_1 + b$, for some $A \in \SL(2,\mathbb R)$ and $b \in \mathbb R^2$.  If $f$ is not an isometry, then at each nonsingular point there is a direction of maximal stretch.  Using the local coordinates we see that the direction of maximal stretch defines a parallel line field on $S \setminus \Sigma(\varphi_1)$.  It follows from the existence of such a parallel line field that the holonomy of $\varphi_1$ is contained in $\{\pm I\} < SO(2)$.  Consequently, $\varphi_1$ and $\varphi_2$ are defined by quadratic differentials differing by an element of $\SL(2,\mathbb R)$ from the usual action of $\SL(2,\mathbb R)$ on the space of quadratic differentials; see \cite{MasTab}, for example.

\begin{proposition} \label{P:affine deformation iff quadratic differential} If $\varphi_1,\varphi_2$ are two affine-equivalent flat metrics which are not equivalent, then $\varphi_1,\varphi_2$ are defined by quadratic differentials in the same $SL(2,\mathbb R)$--orbit. \hfill \qed
\end{proposition}

In \cite{DLR}, $\Flat(S)$ denoted only the space of flat metrics coming from quadratic differentials.  In the current paper (as well as \cite{BL}), the flat metrics coming from quadratic differentials, which we denote $\Flat_2(S) \subset \Flat(S)$ is a (fairly small) subspace (characterized, for example, by Proposition~\ref{P:affine deformation iff quadratic differential} as those admitting affine deformations).  The following theorem, proved in \cite[Lemma 19]{DLR}, gives a useful way of deciding when quadratic differential metrics are affine-equivalent.

\begin{theorem} \label{DLR}
If $\varphi_1$ and $\varphi_2$ are quadratic differential metrics on any surface of finite type, then they are affine-equivalent
if and only if their cylinder curves agree.
\end{theorem}
\begin{remark}  Note that this theorem holds for finite-type surfaces, which allows a finite number of punctures.  We will apply it to a closed surface minus the cone points of a flat metric in the proof of the Support Rigidity Theorem in Section~\ref{S:support-rigidity}.
\end{remark}

When $\varphi_1$ has holonomy with order at least $3$, then by appealing to Proposition~\ref{P:affine deformation iff quadratic differential}, affine equivalence implies the stronger conclusion that $\varphi_1 = \varphi_2$ in $\Flat(S)$.

\subsection{Currents and chains for flat metrics} \label{S:flat liouville}  \label{S:chains section}

The fact that $SL(2,\R)$ orbits in $\Flat_2(S)$  are determined by cylinder sets (Theorem~\ref{DLR} above) 
was used in \cite{DLR} as a step in proving that marked simple closed curves are spectrally rigid over $\Flat_2(S)$; that is, the lengths of those curves entirely determine a quadratic differential metric.  
This rigidity result was generalized in \cite{BL} to  all of $\Flat(S)$, although in this setting all closed curves are needed, not just simple ones.  
In both \cite{DLR} and \cite{BL}, a key object is the Liouville current for a flat metric $\varphi \in \Flat(S)$.  In this 
section we recall some of the key properties of the current, particularly relating to its support.  

In \cite{DLR}, the Liouville current $L_\varphi$ associated to a flat metric $\varphi$ is defined as a kind of average of the measured foliations in all directions and which 
satisfies the geometricity property $\iota(\alpha,L_\varphi) = \ell_\varphi(\alpha)$ for all closed curves $\alpha$.  
For general $\varphi \in \Flat(S)$, the Liouville current $L_\varphi$ is defined in \cite{BL} as the push-forward via the "endpoint map" $\partial_{\tilde \varphi} \colon \mathcal G(\tilde \varphi) \to \mathcal G(\tilde S)$ of a $\pi_1S$--invariant measure on $\mathcal G(\tilde \varphi)$ obtained from Riemannian geometry.  It is also shown there to have the same geometricity property as in \cite{DLR}.
One of the key facts we will need here, proved in  \cite[Proposition~3.4 and Corollary~3.5]{BL}, involves the support of $L_\varphi$.
Recall from \S\ref{S:basic} that the set of basic geodesics $\G^*(\tilde \varphi) \subset \G(\tilde \varphi)$ is the closure of the set of nonsingular geodesics, and that they are completely described
in Proposition~\ref{P:basic}.  We will write $\G^*_{\tilde \varphi} = \partial_{\tilde \varphi}(\mathcal G^*(\tilde \varphi))  \subset \G(\tilde S)$ for the set of endpoints on $S^1_\infty$ of the basic geodesics.

\begin{proposition} \label{P:support_are_basic_endpoints} For any $\varphi \in \Flat(S)$, the support of its Liouville current is precisely given by endpoints of basic geodesics:
$$\supp(L_\varphi) = \G^*_{\tilde \varphi} \subset \G(\tilde S)$$
\end{proposition}

Next we introduce a key tool from \cite{BL} used to study Liouville currents for flat metrics, the set of {\em chains}, 
a technical device to encode information about cone points in terms  of boundary points.  The definition is somewhat technical, and involves an auxiliary countable subset $\Omega \subset \G^*_{\tilde \varphi}$.

\begin{definition} Fix a flat metric $\varphi$ and countable subset $\Omega \subset \G^*_{\tilde \varphi}$.  
Then a {\em $(G^*_{\tilde \varphi},\Omega)$--chain} is a bi-infinite sequence of boundary points 
${\bf{x}}=(\ldots,x_0,x_1,\ldots)\subset S^1_{\infty}$ such that 
\begin{enumerate}[(i)]
\item $\{x_i, x_{i+1}\}\in \G^*_{\tilde \varphi} \setminus \Omega$,
\item $x_i, x_{i+1}, x_{i+2}$ is a counterclockwise-ordered triple of distinct points, and
\item $[\{x_i,x_{i+1}\}, \{x_{i+1},x_{i+2}\}]\cap \G^*_{\tilde \varphi} = \{\{x_i,x_{i+1}\}, \{x_{i+1}, x_{i+2}\}\}$
\end{enumerate} 
for all $i$.  Let the set of $(\G^*_{\tilde \varphi},\Omega)$--chains be denoted $\Chain(\G^*_{\tilde \varphi},\Omega)$.
\end{definition}

\begin{remark} This definition is a more concise version of the one given in \cite{BL} which involved first defining chains using $\G^*_{\tilde \varphi}$ and then introducing the countable set later, and additionally allowed for finite or half-infinite chains.  Furthermore, the notation differs from \cite{BL}, where these were called $(L_\varphi,\Omega)$--chains.  Since the definition depends only on $\supp(L_\varphi) = \G^*_{\tilde \varphi}$, this notation is  more descriptive and will be useful below.
We note that chains may be periodic, and hence may contain only a finite number of boundary points, repeated infinitely often.
\end{remark}

A chain arises naturally from a cone point $\zeta \in \Sigma(\tilde \varphi)$ via families of basic geodesics meeting 
only $\zeta$ and winding around in a manner such that two successive ones are cone-point asymptotic.
In \cite[Proposition~4.1]{BL} it was shown that essentially all chains arise in this way.

The set $\Omega$ is used to discard geodesics meeting more than one cone point.  More precisely, recall that the set $\G^2(\tilde \varphi) \subset \G^*(\tilde \varphi)$ of basic geodesics meeting at least two cone points is countable, by Proposition~\ref{P:basic} (Structure of basic geodesics).  We let $\G^2_{\tilde \varphi} = \partial_{\tilde \varphi} (\G^2(\tilde \varphi)) \subset \G^*_{\tilde \varphi}$ denote the countable set of endpoints of such geodesics, and then require our countable set $\Omega$ to contain $\G^2_{\tilde \varphi}$.

\begin{proposition} \label{P:chains from geodesics}  Suppose $\Omega \subset \G^*_{\tilde \varphi}$ is any countable set containing $\G^2_{\tilde \varphi}$ and ${\bf x} \in \Chain(\G^*_{\tilde \varphi},\Omega)$ is any chain.
Then there exists a unique cone point $\zeta$ and sequence ${\boldsymbol \gamma} = (\ldots,\tilde \gamma_0,\tilde \gamma_1,\ldots) \subset \G^*({\tilde \varphi})$ of geodesics meeting $\zeta$, and no other cone points,  
 such that $\partial_{\tilde \varphi}(\boldsymbol \gamma) = {\bf x}$, i.e., $\partial_{\tilde \varphi}(\tilde\gamma_i)=\{x_i,x_{i+1}\}$
 for each $i$. Moreover, every cone point is related in this way to some chain.
\end{proposition}

From this we obtain a well-defined, surjective map 
 from chains to the set of cone points $\Sigma(\tilde \varphi)$:
\begin{equation} \label{E:chain to cone} \partial^{\#}_{\tilde{\varphi}} \colon \Chain(\G^*_{\tilde \varphi} ,\Omega) \to \Sigma(\tilde{\varphi}). \end{equation}
(See \cite[Lemma~4.2]{BL}.)  Because $\partial_{\tilde \varphi}$ is $\pi_1$--equivariant, so is $\partial^{\#}_{\tilde \varphi}$.

The upshot of this is that when we have two metrics $\varphi_1,\varphi_2 \subset \Flat(S)$ whose Liouville currents have the same support 
$\G^*_{\tilde \varphi_1} = \G^*_{\tilde \varphi_2}$, 
we can take $\Omega$ to be a countable subset containing $\G^2_{\tilde \varphi_1} \cup \G^2_{\tilde \varphi_2}$ so that the resulting chains are the same, producing an identification of the cone points of the two metrics.
The following is essentially a consequence of \cite[Lemma~4.4]{BL} and the proof of \cite[Theorem~5.1]{BL}.

\begin{proposition} \label{P:optimal Omega}
Suppose that $\varphi_1,\varphi_2 \in \Flat(S)$ are two metrics whose Liouville currents have the same support, which we denote $\G^* = \G^*_{\tilde \varphi_1} = \G^*_{\tilde \varphi_2}$. 
Then for any countable set $\Omega \subset \G^*$ containing $\G^2_{\tilde \varphi_1} \cup \G^2_{\tilde \varphi_2}$, we have
\[ \Chain(\G^*,\Omega) = \Chain(\G^*_{\tilde \varphi_1},\Omega) = \Chain(\G^*_{\tilde \varphi_2},\Omega),\]
and there exists a $\pi_1S$--equivariant bijection
\[ \tilde f^0 \colon \Sigma(\tilde \varphi_1) \to \Sigma(\tilde \varphi_2) \]
so that $\tilde f^0 \circ \partial^\#_{\tilde{\varphi}_1} = \partial^\#_{\tilde{\varphi}_2}$.
\end{proposition}

Later we will see that by enlarging $\Omega$ further, we can also make the nonsingular geodesics match up for a pair of metrics $\varphi_1,\varphi_2$ as we did here for chains and cone points.

\subsection{Identifying geodesics} \label{S:convention1}

For two flat metics whose Liouville currents have the same support, we combine facts we know about basic geodesics and flat strips to provide a geometrically useful identification of "most of" the basic geodesics for the two metrics.
Recall that having the same support precisely means that  
$\G^*_{\tilde\varphi_1} = \G^*_{\tilde\varphi_2}$, 
or in other words the basic geodesics have the same endpoints.  Below, we will simply write
\[ \G^*=\G^*_{\tilde\varphi_1} = \G^*_{\tilde\varphi_2},\]
for the mutual support of the pair of Liouville currents.

Metrics in $\Flat(S)$ are only defined up to isotopy.  
In the previous Section~\ref{S:chains section} where chains are defined we recorded the fact that for 
any two metrics $\varphi_1,\varphi_2 \in \Flat(S)$ whose Liouville currents have the same support, 
we can take any countable set   $\Omega$  
containing the boundary pairs of "multiple-cone-point geodesics" $\G^2_{\tilde \varphi_1} \cup \G^2_{\tilde \varphi_2} \subset \G^*$ in either metric, and using the chains, produce a $\pi_1S$--equivariant bijection between the cone points of the two metrics; see Proposition~\ref{P:optimal Omega}.
From arbitrary representatives of the metrics, we can therefore perform "point-pushing" isotopies on the cone points so that the two metrics have exactly the same set of cone points $\Sigma(\tilde \varphi_1) = \Sigma(\tilde \varphi_2)$ in a way that is compatible with the chains (i.e., so that $\partial^\#_{\tilde \varphi_1} = \partial^\#_{\tilde \varphi_2}$).

Second, for the same $\Omega$, we can observe that every pair $\{x,y\} \in \G^* \setminus \Omega$ has exactly one geodesic in its preimage by both $\partial_{\tilde \varphi_1}$ and $\partial_{\tilde \varphi_2}$,
\[ |\partial_{\tilde \varphi_1}^{-1}(\{x,y\})| =1 = |\partial_{\tilde \varphi_2}^{-1}(\{x,y\})|.\]
To see this, note that the only way $\partial_{\tilde \varphi_i}^{-1}(\{x,y\})$ can contain more than one geodesic is if it consists of a (closed) flat strip of parallel geodesics (see \cite[Theorem II.2.13]{BridHaef}).  By the Flat Strip Theorem~\ref{T:flat strips}, such a strip covers a maximal cylinder.  Since the boundary of such a cylinder is a concatenation of saddle connections in $S$, the boundary of the strip itself is a concatenation of saddle connections in $\tilde S$, and hence $\{x,y\} \in \G^2_{\tilde \varphi_i}$, which is a contradiction.

We can then enlarge the countable set $\Omega$ so that for any $\{x,y\} \in \G^* \setminus \Omega$, 
either $x,y$ can be ordered to become part of a bi-infinite (possibly periodic) chain, or else it is the image of nonsingular geodesics in both 
$\tilde \varphi_1$ and $\tilde \varphi_2$.  This is possible because, for each $i= 1,2$ and for each of the countably many cone points $\zeta \in \Sigma(\tilde \varphi_i)$, there are only countably many nonsingular rays emanating from $\zeta$ whose endpoints on $S^1_\infty$ fail to be extendable to a bi-infinite chain.  Thus there are countably many geodesics $\tilde \gamma \in \G^*_{\tilde \varphi_i}$ containing a single cone point whose endpoints are not part of a bi-infinite chain.  So by taking $\Omega$ to be the union of the $\partial_{\tilde \varphi_i}$--images of these, together with $\G^2_{\tilde \varphi_1} \cup \G^2_{\tilde \varphi_2}$, we obtain the desired property.

Therefore, setting $G_i = \partial_{\tilde \varphi_i}^{-1}(\G^* \setminus \Omega)$, for $i=1,2$, 
we can define $g \colon G_1 \to G_2$ to be the bijection
\[ g = \partial_{\tilde \varphi_2}^{-1} \circ \partial_{\tilde \varphi_1}|_{\G^*\setminus \Omega}.\]
That is, $g$ is the $\tilde \varphi_2$--straightening of each $\tilde \varphi_1$--geodesic in $G_1$, and it sends the set of singular geodesics to the set of singular geodesics.  Furthermore, observe that any singular geodesic $\tilde \gamma \in G_i$ containing a cone point $\zeta$ has endpoints in a chain whose $\partial_{\tilde \varphi_i}^\#$--image is in a chain determining $\zeta$.  It follows that $\zeta$ is the unique cone point in a singular geodesic $\tilde \gamma \in G_1$ if and only if it is the unique cone point in $g(\tilde \gamma) \in G_2$.

\begin{proposition} \label{P:g a homeo} 
For any pair of flat metrics with the same support $\G^*$, after discarding a suitable countable set $\Omega\subset\G^*$ the bijection $g \colon G_1 \to G_2$ between the basic geodesics in the two metrics is a homeomorphism. 
Furthermore, $g(\tilde \gamma)$ is singular if and only if $\tilde \gamma$ is.  In the singular case, the unique cone point of $\tilde \gamma$ is also contained in $g(\tilde \gamma)$.
\end{proposition}
\begin{proof} Suppose $\{\tilde \gamma_n\}_{n = 1}^\infty \subset G_1$ is a sequence such that $\tilde \gamma_n \to \tilde \gamma \in G_1$.  To prove $g(\tilde \gamma_n) \to g(\tilde \gamma)$, it suffices to prove that every subsequence of $\{g(\tilde \gamma_n)\}$ has a subsequence that converges to $g(\tilde \gamma)$.  For this, observe that $\{\partial_{\tilde \varphi_1} (\tilde \gamma_n)\}_{n=1}^\infty \cup \{ \partial_{\tilde \varphi_1} (\tilde \gamma)\}$ is compact subset of $\G(\tilde S)$.  Therefore, $\{g(\tilde \gamma_n)\}_{n=1}^\infty \cup \{g(\tilde \gamma)\}$ is precompact in $\G(\tilde \varphi_2)$, and since $\G^*(\tilde \varphi_2)$ is a closed subspace containing $G_2$, our subset is precompact in $\G^*(\tilde \varphi_2)$.  For any subsequence of $\{g(\tilde \gamma_n)\}$ we may choose a convergent subsequence $\{g(\tilde \gamma_{n_k})\}_{k=1}^\infty$ with $g(\tilde \gamma_{n_k}) \to \tilde \gamma' \in \G^*(\tilde \varphi_2)$.  Then, by definition of $g$ (and the fact that limits in the Hausdorff space $\G(\tilde S)$ are unique) we have
\[ \partial_{\tilde \varphi_2}(\tilde \gamma') = \lim_{k \to \infty} \partial_{\tilde \varphi_2}(g(\tilde \gamma_{n_k})) = \lim_{k \to \infty} \partial_{\tilde \varphi_1}(\tilde \gamma_{n_k}) =  \partial_{\tilde \varphi_1}(\tilde \gamma).\]
Thus, $\tilde \gamma' = g(\tilde \gamma)$ since $\partial_{\tilde \varphi_2}$ is one-to-one on $G_2$.  Therefore, $g$ is continuous.  Reversing the roles of $G_1$ and $G_2$, we see that $g^{-1}$ is also continuous, and hence $g$ is a homeomorphism.
\end{proof}

We now adopt the machinery needed to define this homeomorphism as a convention for the remainder of the paper.  
For two flat metrics $\varphi_1,\varphi_2$ with equal support, we will be able to unambiguously write $\G^*$ for their mutual support
as pairs of endpoints; $\Omega\subset\G^*$ for a countable set as described above; $G_i= \partial_{\tilde\varphi_i}(\G^*\setminus \Omega)$
for the remaining basic geodesics; $\widetilde\Sigma=\Sigma(\tilde\varphi_1)=\Sigma(\tilde\varphi_2)$ for the cone points;
and $g:G_1\to G_2$ for the homeomorphism identifying the basic geodesics.

\section{Support Rigidity Theorem}\label{sec_support}

In this section we will prove our main theorem about Liouville currents of flat metrics.

\begin{supporttheorem}
Suppose $\varphi_1,\varphi_2$ are two unit-area flat metrics whose Liouville currents have the same support,
$\supp(L_{\varphi_1}) = \supp(L_{\varphi_2})$.  Then $\varphi_1,\varphi_2$ differ by an affine deformation, up to isotopy.  If either metric has holonomy of order greater than $2$ (i.e., is not induced by a quadratic differential, or does not support a foliation by straight lines), 
then equal support implies that $\varphi_1$ and $\varphi_2$ differ by isometry, isotopic to the identity.
\end{supporttheorem}

For this entire section, we assume that $\varphi_1$ and $\varphi_2$
are flat metrics with equal support, and we 
follow the conventions specified above to identify their cone points and their basic geodesics.

\subsection{Cone point partitions} \label{S:partitions}

For each $i = 1,2$, any geodesic $\tilde \gamma \in G_i$ divides $\tilde S$ into two {\em half-planes}, $\mathcal H_i^{\pm}(\tilde \gamma)$, which are the closures of the connected components of the complement of $\tilde \gamma$.  When $\tilde \gamma$ is a nonsingular geodesic, each cone point lies in exactly one of these two half-planes, and thus $\tilde \gamma$ determines a partition of $\widetilde \Sigma$ into two disjoint subsets depending on which side of $\tilde \gamma$ the point lies.  When $\tilde \gamma$ is singular, it contains exactly one cone point, and makes cone angle $\pi$ on one side.  We then declare the cone point to lie on the side opposite the one in which it makes an angle $\pi$, and so $\tilde \gamma$ also determines a partition of $\widetilde \Sigma$ into two disjoint subsets.

\begin{lemma}[Cone point partitions are well-defined]
\label{geodcomb}
For any  $\tilde \gamma \in G_1$, the geodesic $\tilde\gamma$ and its image $g(\tilde\gamma)$ define the same partition of $\widetilde \Sigma$.
\end{lemma}

\begin{proof}
First, suppose that $\tilde \gamma \in G_1$ is nonsingular, and fix a $\zeta \in \widetilde \Sigma$.   After orienting $\tilde \gamma$, let us suppose $\mathcal H_1^+(\tilde \gamma)$ is the right half-plane and $\mathcal H^-_1(\tilde \gamma)$ is the left half-plane.
The orientation of $\tilde \gamma$ induces an orientation of $g(\tilde \gamma)$ and hence a choice of right and left half-planes $\mathcal H_2^+(g(\tilde \gamma))$ and $\mathcal H_2^-(g(\tilde \gamma))$.
The endpoints $\partial_{\tilde \varphi_1}(\tilde \gamma) = \partial_{\tilde \varphi_2}(g(\tilde \gamma))$ divide $S^1_\infty$ into two components, $S^{1\pm}_\infty(\tilde \gamma)$,
which are the boundaries at infinity of $\mathcal H_1^\pm(\tilde \gamma)$, as well as that of $\mathcal H_2^\pm(g(\tilde \gamma))$.

We will show $\zeta \in \mathcal H_1^+(\tilde \gamma) \implies \zeta \in \mathcal H_2^+(\tilde \gamma)$.
To prove this, observe that there is a singular geodesic $\tilde \delta \in G_1$ containing $\zeta$ and contained in the interior of $\mathcal H^+_1(\tilde \gamma)$.  Such a geodesic can be obtained by first parallel translating the direction of $\tilde\gamma$ to $\zeta$ along a geodesic segment connecting a point of $\tilde \gamma$ to $\zeta$ (making appropriate choices at any cone points encountered along this geodesic segment), then perturbing slightly to avoid the countably many geodesics through $\zeta$ not in $G_1$.  
Then  $\zeta \in g(\tilde \delta)$.  On the other hand, $\tilde \gamma$ and $g(\tilde \gamma)$ have the same endpoints $\partial_{\tilde \varphi_1}(\tilde \delta) = \partial_{\tilde \varphi_2}(g(\tilde \delta))$, and these lie in $S^{1+}_\infty(\tilde \gamma)$.  If $\zeta \in \mathcal H^-_2(g(\tilde \gamma))$, then $g(\tilde \delta)$ would have to cross $g(\tilde \gamma)$ twice, creating a bigon; see Figure~\ref{bigonpic}.  This contradiction shows that $\zeta \in \mathcal H_2^+(g(\tilde \gamma))$, as required.

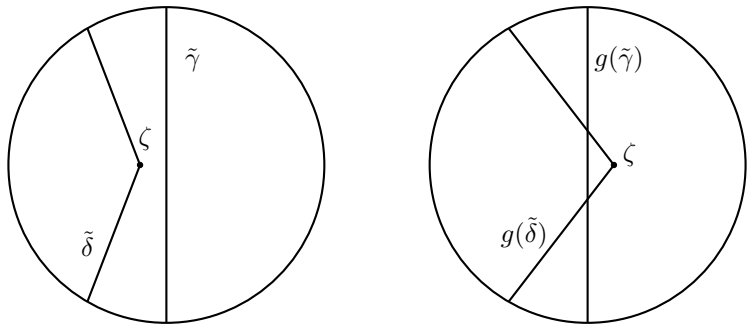
\begin{figure}[h]
\begin{center}
\begin{tikzpicture}[scale = .7]
\draw[thick] (0,0) circle (3);
\draw[thick] (0,3) -- (0,-3);
\draw[thick] ({3*cos(120)},{3*sin(120)}) -- (-.5,0) -- ({3*cos(-120)},{3*sin(-120)});
\node at (.5,2) {$\tilde \gamma$};
\node at (-1.5,-1.5) {$\tilde \delta$};
\node at (-.4,.5) {$\zeta$};
\draw[thick] (8,0) circle (3);
\draw[thick] (8,3) -- (8,-3);
\draw[thick] ({8+3*cos(120)},{3*sin(120)}) -- (8.5,0) -- ({8+3*cos(-120)},{3*sin(-120)});
\node at (8.6,2) {$g(\tilde \gamma)$};
\node at (6.8,-1.3) {$g(\tilde \delta)$};
\node at (8.8,.2) {$\zeta$};
\draw[fill=black] (-.5,0) circle (.05cm);
\draw[fill=black] (8.5,0) circle (.05cm);
\end{tikzpicture}
\caption{Switching sides creates a bigon.}
\label{bigonpic}
\end{center}
\end{figure}

Next, observe that if $\zeta \in \widetilde \Sigma$ is on the left-hand side of a sequence of $\tilde \varphi_i$--geodesics $\{\tilde \gamma_n\} \subset G_i$ (with respect to a choice of orientations) and $\tilde \gamma_n \to \tilde \gamma \in G_i$ (limiting as oriented geodesics), then $\zeta$ is also on the left-hand side of $\tilde \gamma$, for $i=1,2$.  Now suppose that $\zeta \in \widetilde \Sigma$ is any cone point and $\tilde \gamma \in G_1$ is a geodesic for which $\zeta$ is on the left-hand side.  Choose a sequence of nonsingular geodesics $\tilde \gamma_n \in G_1$ limiting to $\tilde \gamma$, such that $\zeta$ is on the left-hand side of $\tilde \gamma_n$ for all $n$: this is easy to do if $\zeta$ is not on $\tilde \gamma$ (then any sequence will have a tail that has this property).  If, on the other hand, $\zeta \in \tilde \gamma$, then it makes angle $\pi$ on the right, and any sequence $\{\tilde \gamma_n\} \subset G_1$ of nonsingular geodesics limiting to $\tilde \gamma$ from the right near $\zeta$ will have $\zeta$ on the left-hand side (to find such a sequence, take any sequence of points approaching $\zeta$ in the interior of $\mathcal H^+_1(\tilde \gamma)$ whose tangent vectors limit to the direction of $\tilde \gamma$ and so that there are nonsingular geodesics through those points which lie in $G_1$).  Now since $g$ is a homeomorphism, it follows that $g(\tilde \gamma_n) \to g(\tilde \gamma)$.  Since $\tilde \gamma_n$ are all nonsingular, $\zeta$ is on the left-hand side of $g(\tilde \gamma_n)$ by the first part of the proof, and hence also on the left-hand side of $g(\tilde \gamma)$.
\end{proof}

\subsection{Identifying triangulations} \label{S:triangulations}

Recall that given $\varphi \in \Flat(S)$, a $\tilde{\varphi}$--Euclidean triangle is a triangle with cone point vertices, 
$\tilde{\varphi}$--saddle connection sides, and  no cone points in its interior. 

\begin{lemma} \label{L:saddle and euclidean preserved}
Every pair $x,y\in \tilde\Sigma$ that determines a saddle connection for $\tilde{\varphi}_1$ also determines 
a saddle connection for $\tilde{\varphi}_2$; this defines a bijection $g_{sc} \colon G_{1,sc} \to G_{2,sc}$ between all $\tilde \varphi_1$--saddle connections and all $\tilde \varphi_2$--saddle connections.  Furthermore, $x,y,z \in \widetilde \Sigma$ are vertices of a positively oriented $\tilde \varphi_1$--Euclidean triangle if and only if they are vertices of a positively oriented $\tilde \varphi_2$--Euclidean triangle. 
\end{lemma}

\begin{proof} 
To prove the first statement, suppose $x, y$ determines a saddle connection $\delta$ for $\tilde{\varphi}_1$. This means that there is a $\tilde{\varphi_1}$--geodesic segment connecting $x$ and $y$ that does not pass through any other cone points. Now consider the unique $\tilde{\varphi}_2$--geodesic path between $x$ and $y$. Suppose there is a cone point $z$ on this path. We claim that it is possible to separate $z$ from $x$ and $y$ by  a $\tilde{\varphi}_1$--geodesic, $\tilde \gamma \in G_1$ as in Lemma \ref{geodcomb}. Note that this is impossible in the $\tilde{\varphi}_2$ metric, since this would create a geodesic bigon (see Figure \ref{bigon}).

To create the desired separating geodesic, consider the $\tilde\varphi_1$--geodesic segment from $z$ to $x$ and the $\tilde\varphi_1$--geodesic segment from $z$ to $y$. There must be sub-segments of each of these that do not intersect $\delta$. This is because in the $\tilde\varphi_1$ metric, $z$ is not on the saddle connection $\delta$. We denote these sub-segments as $\alpha$ and $\beta$ in Figure \ref{bigon}, below. Any geodesic not passing through $x$, $y$ or $z$ and transversally intersecting both $\alpha$ and $\beta$ will separate $z$ from $x$ and $y$. It is possible to find the required geodesic $\tilde \gamma \in G_1$ by perturbing an arbitrary such $\tilde \varphi_1$--geodesic to one in $G_1$.
This completes the proof of the first statement.

\begin{figure}[h]
\centering
\begin{tikzpicture}
\node (img) at (0,0) {\includegraphics[width=5in]{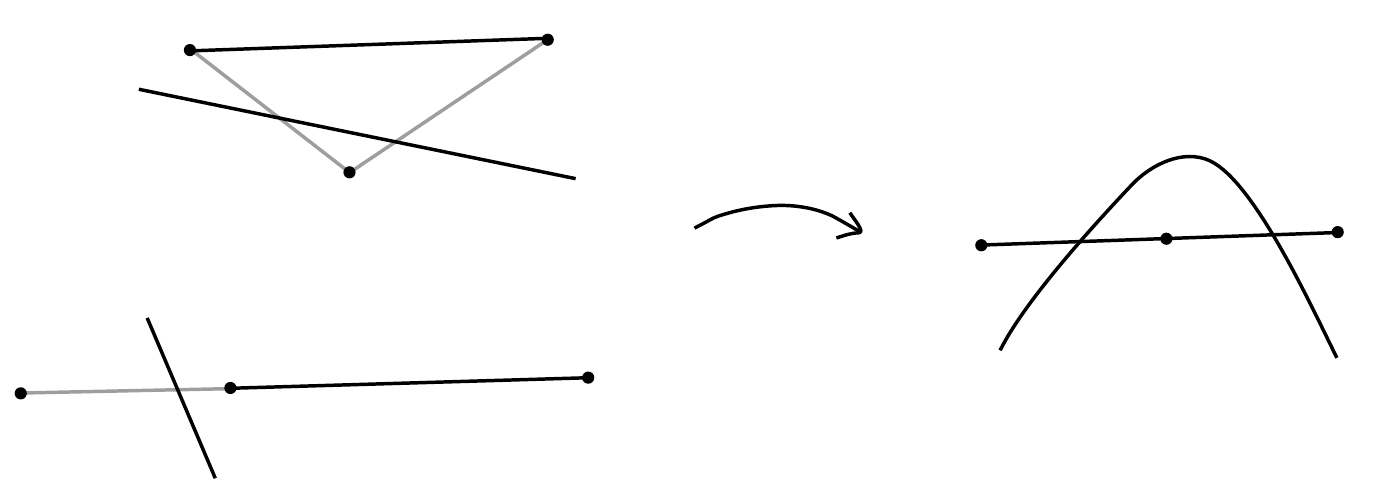}};
\node at (23pt, 18pt) {$g$};
\node at (72pt, -5pt) {$x$};
\node at (175pt, -2pt) {$y$};
\node at (128pt, -4pt) {$z$};
\node at (120pt, 32pt) {$g(\tilde\gamma)$};
\node at (-131pt, 59pt) {$x$};
\node at (-30pt, 61pt) {$y$};
\node at (-88pt, 14pt) {$z$};
\node at (-102pt, 25pt) {$\alpha$};
\node at (-50pt, 39pt) {$\beta$};
\node at (-149pt, 45pt) {$\tilde\gamma$};
\node at (-83pt, 61pt) {$\delta$};
\node at (-118pt, -29pt) {$x$};
\node at (-20pt, -28pt) {$y$};
\node at (-180pt, -32pt) {$z$};
\node at (-152pt, -43pt) {$\alpha=\beta$};
\node at (-145pt, -10pt) {$\tilde\gamma$};
\node at (-68pt, -41pt) {$\delta$};
\node at (-78pt, -8pt) {or};

\end{tikzpicture}
\caption{Failure to preserve saddle connections creates a bigon.}
\label{bigon}
\end{figure}
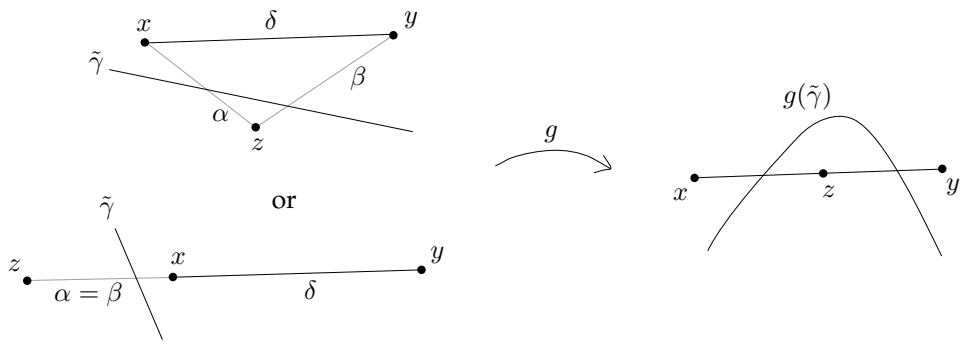

To prove the second statement, suppose that $x$, $y$, $z$ are the vertices of a $\tilde{\varphi}_1$-Euclidean triangle $\tilde T$. From the above, there are $\tilde{\varphi}_2$--saddle connections forming a complete graph on $x$, $y$, $z$. Further, these saddle connections bound a connected component, which we call $\tilde T'$. We need to show that $\tilde T'$ does not contain any cone points, hence $x$, $y$, $z$ form a $\tilde{\varphi}_2$-Euclidean triangle, and that this triangle $\tilde T'$ has the same orientation as $\tilde T$. 

Let $\alpha$ be the saddle connection between $x$ and $z$. There is a sequence of nonsingular geodesics $\{\tilde \alpha_n\}$ in $G_1$, intersecting $\tilde T'$ which limits to $\alpha$ in $\tilde T$. Let $\beta$ and $\delta$ be the saddle connections between $y$ and $z$ and between $x$ and $z$, respectively. Define similarly sequences $\{\tilde \beta_n\}$ and $\{\tilde \delta_n\}$. For large enough $N$, $\tilde \alpha_N$ defines a partition that separates $z$ from $x$ and $\tilde \delta_N$ a partition that separates $z$ from $y$. Consider the intersection of corresponding two half-planes containing $z$. The boundary of this region determines an interval $I_x$ on $S^1_{\infty}$. Similarly $\alpha _N, \tilde \beta_N$, $z$ and $\tilde \beta_N, \tilde \delta_N, y$ determine intervals $I_y$ and $I_z$ on $S_1$, respectively (see the left-hand side of Figure \ref{F:orientation triangle}).  Give the triangle bounded by $\tilde \alpha_N$, $\tilde \beta_N$, and $\gamma_N$ the same orientation as $\tilde T'$. Note that the orientation of the intervals $I_x, I_y, I_z$ along $S^1_{\infty}$ determine the orientation of the triangle (although we note that the intersection pattern of the geodesics in the $\tilde \varphi_2$--metric is not a priori the same; see the right-hand side of Figure~\ref{F:orientation triangle}). Moreover, the 6 intervals determined by $I_x, I_y, I_z$ and their complements partition $S^1_{\infty}$ in such a way that the $\tilde \varphi_1$--geodesics obtained by straightening $\tilde \alpha_N$ and $\tilde \beta_N$ separate vertex $x$ from the other vertices of $\tilde T$.  Similarly the straightening of $\tilde \alpha_N$ and $\tilde \delta_N$ separate vertex $z$ from $y$ and $x$ and the straightening of $\tilde \beta_N$ and $\tilde \delta_N$ separate vertex $y$ from $z$ and $x$. This results in a triangle whose orientation must agree with both $\tilde T'$ and $\tilde T$ and hence the orientation of $\tilde T$ and $\tilde T'$ must agree. 
 
Finally, if $w$ is a cone point inside $\tilde T'$, we similarly approximate the side of $\tilde T'$ by geodesics $\tilde \alpha_N, \tilde \beta_N, \tilde \delta_N \in G_2$ which together separate this cone point from the cone points $x, y, z$ in the $\tilde{\varphi_2}$-metric, and hence, since $g^{-1}: G_2 \to G_1$ is a homeomorphism, induces such a separation in the $\tilde\varphi_1$-metric as well, contradicting that $\tilde T$ is a $\tilde\varphi$-Euclidean triangle.\end{proof}

\begin{figure}[h]
\begin{center}
\begin{tikzpicture}[scale = 0.8]
\draw (0,0) -- (3,0) -- (1,2) -- (0,0);
\node at (2.4,1) {$\alpha$};
\node at (1.5,-0.4) {$\beta$};
\node at (0.2,1) {$\delta$};
\node at (3.2,-.7) {$\tilde \alpha_n$};
\node at (-1,.5) {$\tilde \beta_n$};
\node at (-.2,-1.3) {$\tilde \delta_n$};
\node at (-0.2,-0.2) {$y$};
\node at (3.3,-0.1) {$x$};
\node at (1,2.2) {$z$};
\node at (1.3,0.5) {$w$};
\draw[fill=black] (0,0) circle (.05cm);
\draw[fill=black] (3,0) circle (.05cm);
\draw[fill=black] (1,2) circle (.05cm);
\draw[fill=black] (1.3,0.7) circle (.05cm);
\node at (-2,-1.5) {$I_y$};
\node at (4.3, -0.5) {$I_x$};
\node at (0.7, 3.3) {$I_z$};
\draw (1,0) circle (3);
\draw[dashed](-2,0.1)--(4,0.1);
\draw[dashed](-1.02,-2.2)--(1.56,2.95);
\draw[dashed](0.03,2.8)--(3.9,-1);
\draw (10,0) circle (3);
\draw[dashed](7,0.1) .. controls(10,2) .. (13,0.1);
\draw[dashed] (7.98,-2.2) .. controls (12,0) .. (10.56,2.95);
\draw[dashed] (9.03,2.8) .. controls (9,-1) .. (12.9,-1);
\node at (7,-1.5) {$I_y$};
\node at (13.3, -0.5) {$I_x$};
\node at (9.7, 3.3) {$I_z$};
\end{tikzpicture}
\caption{{\bf Left-hand side:} Approximating the sides $\alpha, \beta, \gamma$ of $\tilde T'$ with nonsingular geodesics and the resulting intervals $I_x, I_y, I_z$ on $S_1$.  {\bf Right-hand side:} Potential change of intersection pattern.}
\label{F:orientation triangle}
\end{center}
\end{figure}
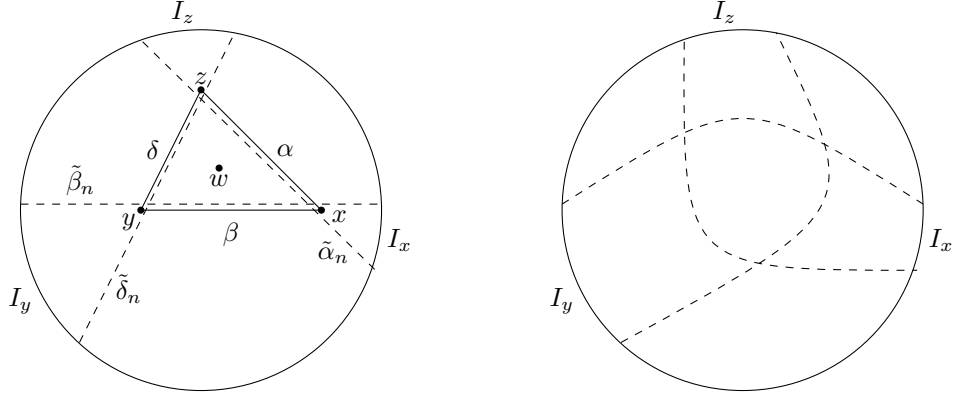

In fact, $g_{sc}$ also defines a bijection between all $\tilde{\varphi_1}$--geodesics that are concatenation of saddle connections and the set of all such $\tilde{\varphi_2}$--geodesic concatenations:

\begin{proposition} \label{P:saddle concatenation} Consider a $\tilde \varphi_1$--geodesic  
 $\tilde \gamma_1 = \cdots \delta_0 \delta_1 \cdots$ which is a 
finite, infinite, or bi-infinite concatenation of saddle connections.
Then $\tilde \gamma_2 = \cdots g_{sc}(\delta_1) g_{sc}(\delta_2) \cdots$
is a $\tilde \varphi_2$--geodesic.  We write $\tilde \gamma_2 = g(\tilde \gamma_1)$.
\end{proposition}

\begin{proof}  
Recall that a concatenation of saddle connections is a geodesic segment if and only if the angle between any two consecutive saddle connections measure at least $\pi$ on either side. Let $\tilde \gamma_1 = \cdots \delta_0 \delta_1 \cdots$ be a $\tilde \varphi_1$--geodesic segment which is a concatenation of saddle connections. Then by Lemma \ref{L:saddle and euclidean preserved}, $\tilde \gamma_2= \cdots g_{sc}(\delta_1) g_{sc}(\delta_2) \cdots$ is a concatenation of $\tilde {\varphi_2}$--saddle connections. Suppose $g_{sc}(\delta_i)$ is a saddle connection between cone points $y$ and $x$ and $g_{sc}(\delta_{i+1})$ a saddle connection between $x$ and $z$  and that they meet at $x$ in an angle less than $\pi$ on one side. Then, similar to the proof of Lemma \ref{L:saddle and euclidean preserved}, there is a nonsingular $\tilde{\varphi_2}$--geodesic that induces a partition separating $x$ from $y$ and $z$ (see Figure \ref{F:concatenation}). But since $\delta_i$ and $\delta_{i+1}$ make up a geodesic segment in the $\tilde{\varphi_1}$-metric, this gives us a $\tilde{\varphi_1}$--geodesic bigon, a contradiction. Hence the angle between two consecutive saddle connections must be at least $\pi$ on both sides, and hence $ \tilde \gamma_2$ is a $\tilde \varphi_2$--geodesic segment. 
\end{proof}

\begin{figure}[h]
\begin{center}
\begin{tikzpicture}[scale = 0.8]
\draw (0,0) -- (-1.5,2) -- (0,4);
\draw[fill=black] (0,0) circle (.05cm);
\draw[fill=black] (-1.5,2) circle (.05cm);
\draw[fill=black] (0,4) circle (.05cm);
\draw (-1.2,1.6) arc (2:3.4:30cm);
\draw[dashed] (-0.7,4.5)--(-0.3,-0.4);
\node at (0.3,0) {$y$};76
\node at (-1.8,2) {$x$};
\node at (0.3,4) {$z$};
\node at (-0.9,2) {$\theta$};
\node at (-1,4.3) {$\alpha$};
\end{tikzpicture}
\caption{If the angle $\theta$ between two consecutive saddle connections measures less than $\pi$ there is a nonsingular geodesic $\alpha$ separating $x$ from $y$ and $z$.}
\label{F:concatenation}
\end{center}
\end{figure}
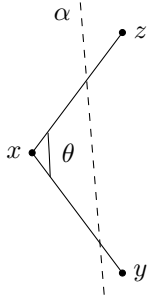

Consider now any $\varphi_1$--triangulation $\mathcal T$ of $S$.  Recall that this means that the vertex set is precisely $\Sigma$ and each triangle is a $\varphi_1$--Euclidean triangle.
We lift to a $\pi_1S$--invariant triangulation $\widetilde{\mathcal T}$ of $\tilde S$.  By Lemma~\ref{L:saddle and euclidean preserved}, we can extend the identity from $\widetilde \Sigma$ to itself to a $\pi_1S$--equivariant map
\[ \tilde f_{\mathcal T} \colon \tilde S \to \tilde S \]
which sends each triangle in $\tilde{\mathcal T}$ to a Euclidean $\tilde \varphi_2$--triangle, and is orientation preserving and affine on the corresponding $2$--simplices.  This map descends to a map $f_{\mathcal T} \colon S \to S$ which is also affine on each $2$--simplex in ${\mathcal T}$.  Since the orientation on each triangle is preserved by $f_{\mathcal T}$, and $f_{\mathcal T}$ has degree $1$, it follows that $f_{\mathcal T}$ is a homeomorphism.  Our goal is to show that, when $\varphi_1$ has holonomy greater than 2, this map is in fact an isometry.  Since the areas of $\varphi_1$ and $\varphi_2$ are both $1$,  if we prove that the interior angles of the triangles in ${\mathcal T}$ are preserved by $f_{\mathcal T}$, then $f_{\mathcal T}$ must be an isometry.  

Adjusting $\varphi_2$ by an isotopy, we can assume that the $\varphi_1$--saddle connections which are the edges of $\mathcal T$ are also $\varphi_2$--saddle connections.  In particular, $\mathcal T$ is also a $\varphi_2$--triangulation.  Moreover, through a further isotopy, we can assume that $f_\mathcal{T}$ and $\tilde{f}_\mathcal{T}$ are the identity on each of the edges in the triangulation. This is convenient to assume throughout the following, and so we make the following additional assumption.
In addition to the conventions adopted at the end of \S\ref{S:convention1} to identify $G_1$ and $G_2$ homeomorphically, 
we further assume that we have fixed a $\varphi_1$--triangulation $\mathcal T$ which is also a $\varphi_2$--triangulation, and let 
$\widetilde{\mathcal T}$ be its lift to $\tilde S$. We also assume that the identity on $(S,\varphi_1) \to (S,\varphi_2)$ is affine on each triangle. 

For each triangle $T$ of $\mathcal T$ and each $i = 1,2$, choose an isometry from $T$ (with its induced path metric) to a Euclidean triangle with respect to $\varphi_i$ (as noted in \S\ref{S:basic}, this may not be well-defined on the vertices).  Using this, we identify the unit tangent space at each point of each triangle (other than the vertices) with $S^1$.  Given $\theta \in S^1$ and a triangle $T$ of $\mathcal T$, we can then unambiguously refer to the $\varphi_i$--direction $\theta$ of $T$.   We let $S^1_i(T)$ denote the $\varphi_i$--directions of $T$.

We do the same for triangles of $\tilde{\mathcal{T}}$ in $\tilde S$, choosing isometries from triangles in $\tilde{\mathcal{T}}$ to Euclidean triangles by composing the covering map $p$ with the isometries of the corresponding triangles in $\mathcal{T}$. Observe that by construction, for any $\tilde T$ in $\widetilde{\mathcal T}$ and $x \in \tilde T$, the derivative $dp_x$ sends $\tilde \varphi_i$--direction $\theta$ to $\varphi_i$--direction $\theta$ in $T = p(\tilde T)$. Thus, for any triangle $\tilde T$ in $\widetilde{\mathcal T}$ with $T = p(\tilde T)$, we use this to identify $S^1_i(T) = S^1_i(\tilde T)$.

With this convention, we are now able to talk about common triangulations and associated directions in a consistent way, for 
example in proving the following corollary to Proposition \ref{P:saddle concatenation}.

\begin{corollary} \label{C:common cylinders}Two currents with the same support have the same cylinder curves: $\cyl(\varphi_1) = \cyl(\varphi_2)$. Moreover, we can homotope any $\varphi_1$-cylinder curve to a $\varphi_2$-cylinder curve in $S\setminus\Sigma$. 
\end{corollary}

\begin{proof} Lift any $\varphi_1$--cylinder to an infinite strip in $\tilde S$.  The boundary is a pair of bi-infinite geodesics, $\tilde \gamma^\pm$ which are each concatenations of $\tilde \varphi_1$--saddle connections.  Moreover, these remain a bounded Hausdorff distance apart (Hausdorff distance equal to the width of the strip, in fact).  Now observe that $g(\tilde \gamma^\pm)$ is a pair of asymptotic $\tilde \varphi_2$--geodesics which are concatenations of $\tilde \varphi_2$--saddle connections.  These must bound an infinite strip, invariant under a cyclic group (since the same is true of the strip bounded by $\tilde \gamma^\pm$), and hence the quotient is a $\varphi_2$--cylinder whose core curve is homotopic to the core curve of the $\varphi_1$--cylinder.  Thus $\cyl(\varphi_1) \subset \cyl(\varphi_2)$. A symmetric argument proves the reverse containment. 

To see that the cylinder curves are homotopic in $S\setminus\Sigma$ we fix a triangulation $\mathcal{T}$ of $S$. Let $\tilde{\gamma}$ be a $\tilde{\varphi}_1$--geodesic contained in some $\tilde{\varphi}_1$--strip. Note that any $\tilde{\varphi}_2$--geodesic in the corresponding $\tilde{\varphi}_2$--strip defines the same partition of $\tilde{\Sigma}$ as $\tilde{\varphi}_1$. This follows by approximating the $\tilde{\varphi}_1$--geodesic by nonsingular geodesics $\tilde{\gamma}_n$ in $G_1$ and applying $g$ to these. Each $g(\tilde{\gamma}_n)$ define the same partition as $\tilde{\varphi}_1$. By passing to a subsequence,  $g(\tilde{\gamma}_n)$ converges to a $\tilde{\varphi}_2$ geodesic in the $\tilde{\varphi}_2$--strip. However, all $\tilde{\varphi}_2$ geodesics in the $\tilde{\varphi}_2$--strip determine the same partition. Now, take a nonsingular $\varphi_1$-representative $\gamma_1$ of a cylinder curve $\gamma$. The ordered and oriented edges of $\mathcal{T}$ crossed by $\gamma$ determine the homotopy class of $\gamma$ in $S\setminus\Sigma$. However, a nonsingular $\varphi_2$--geodesic representative $\gamma_2$ of the homotopy class of $\gamma$ must cross the same set of ordered and oriented edges, since their lifts $\tilde{\gamma}_1$ and $\tilde{\gamma}_2$ to $\tilde{S}$ determine the same partition of $\tilde\Sigma$ in $\tilde{S}$. Hence $\gamma_1$ and $\gamma_2$ must be homotopic in $S\setminus\Sigma$. 
\end{proof}

\subsection{Parallelism preserved} \label{S:split strips}

Suppose $\tilde T$ is any triangle in  $\widetilde{\mathcal T}$, $i \in \{1,2\}$, and $\tilde \gamma \in G_i$ is a geodesic nontrivially intersecting the interior of $\tilde T$.  Define $\G^*(\tilde T,\tilde \gamma)$ to be the closure of the set of geodesics in $\G^*(\tilde \varphi_i)$ intersecting $\tilde T$ in an arc parallel to the arc of intersection $\tilde \gamma \cap \tilde T$.  We say that $\tilde \gamma$ is {\em $\tilde \varphi_i$--generic} for $\tilde T$ if $\G^*(\tilde T,\tilde \gamma) \subset G_i$. 

Before stating the main technical result of this section, we describe the  structure of the set $\G^*(\tilde T,\tilde \gamma)$ for the triangle $\tilde T$, classifying the possibilities for
how these geodesics parallel to $\tilde \gamma$ hit a particular side $\delta$ of each triangle, and allowing us to essentially parametrize the leaves of this foliation by where it hits $\delta$.

\begin{lemma} \label{L:generic triangle foliation} Fix $i \in \{1,2\}$,  a triangle $\tilde T$
in $\widetilde{\mathcal T}$, and  a $\tilde \varphi_i$--generic geodesic $\tilde \gamma \in G_i$.   
Then the intersection of the geodesics of $\G^*(\tilde T,\tilde \gamma)$ with $\tilde T$ defines a foliation by parallel geodesic segments meeting each side of $\tilde T$ transversely.  One side $\delta$ of $\tilde T$ nontrivially intersects every geodesic in $\G^*(\tilde T,\tilde \gamma)$, 
defining a function
\[ h \colon \G^*(\tilde T,\tilde \gamma) \to \delta\]
that sends each geodesic to its point of intersection with $\delta$.  For any $x \in \delta$, the preimage $h^{-1}(x)$ consists of either (1) one nonsingular geodesic, (2) one singular geodesic through one of the two endpoints of $\delta$, or (3) two cone-point asymptotic geodesics meeting $\tilde T$ in the same geodesic arc.
\end{lemma}

\begin{figure}[h]
\begin{center}
\begin{tikzpicture}[scale = 1]
\draw [draw=gray, fill=gray, opacity=.4] (0,0) -- (2,4) -- (4,2);
\draw [ultra thick] (0,0) -- (2,4) -- (4,2) -- (0,0);
\draw (-1,4) .. controls (3,4) .. (5,4.1);
\draw (-1,3.8) .. controls (3,3.8) .. (5,3.9);
\draw (-1,3.6) .. controls (3,3.6) .. (5,3.7);
\draw (-1,3.4) .. controls (3,3.4) .. (5,3.5);
\draw (-1,3.2) .. controls (3,3.2) .. (5,3.3);
\draw (-1,3) .. controls (3,3) .. (5,3.1);
\draw (-1,2.8) .. controls (3,2.8) .. (5,2.9);
\draw (-1,2.6) .. controls (3,2.6) .. (5,2.7);
\draw (-1,2.4) .. controls (3,2.4) .. (5,2.5);
\draw (-1,2.2) .. controls (4,2.2) .. (5,2.3);
\draw (-1,2) -- (4,2) -- (5,2.1);
\draw (4,2) -- (5,1.9);
\draw (-1,1.8) .. controls (4,1.8) .. (5,1.7);
\draw (-1,1.6) .. controls (4,1.6) .. (5,1.5);
\draw (-1,1.4) .. controls (4,1.4) .. (5,1.3);
\draw (-1,1.2) .. controls (4,1.2) .. (5,1.1);
\draw (-1,1) .. controls (4,1) .. (5,.9);
\draw (-1,.8) .. controls (4,.8) .. (5,.7);
\draw (-1,.6) .. controls (4,.6) .. (5,.5);
\draw (-1,.4) .. controls (4,.4) .. (5,.3);
\draw (-1,.2) .. controls (4,.2) .. (5,.1);
\draw (-1,0) .. controls (4,0) .. (5,-.1);
\draw[fill=black] (0,0) circle (.05cm);
\draw[fill=black] (2,4) circle (.05cm);
\draw[fill=black] (4,2) circle (.05cm);
\draw[draw=white, fill=white] (3,3.5) circle (.2cm);
\node at (3,3.5) {$\tilde T$};
\draw[draw=white, fill=white] (.7,2.21) circle (.17cm);
\node at (.7,2.21) {$\delta$};
\end{tikzpicture}
\caption{The foliation of a triangle $\tilde T$ from a $\tilde \varphi_i$--generic geodesic $\tilde \gamma$.  The geodesics on the top and bottom of this figure are the unique geodesics in $\G^*(\tilde T,\tilde \gamma)$ which do not meet the interior of $\tilde T$.  The side $\delta$ meets every geodesic in $\G^*(\tilde T,\tilde \gamma)$.}
\label{F:triangle foliation}
\end{center}
\end{figure}
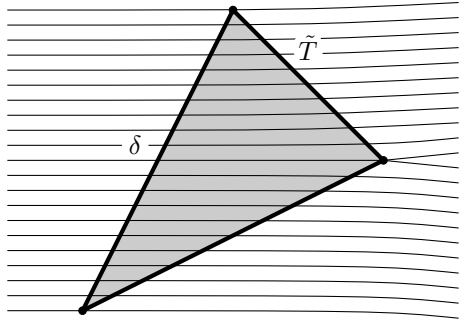

\begin{proof} The  geodesics of $\G^*(\tilde T,\tilde \gamma)$ intersect $\tilde T$ in the leaves of a parallel geodesic foliation by definition.  To see that this foliation meets the sides transversely, we observe that otherwise one of the sides is part of a leaf.  Since such a leaf is a geodesic that contains both endpoints of that side, hence two cone points, it is not in $G_i$, contradicting the $\tilde \varphi_i$--genericity assumption.  

To find the unique side $\delta$ of $\tilde T$ intersecting every geodesic in $\G^*(\tilde T,\tilde \gamma)$, we consider an isometry from $\tilde T$ (with the $\tilde \varphi_i$--metric) to a triangle in $\mathbb R^2$ for which the foliation is horizontal.  The unique side connecting the highest and lowest vertices is (the image of) $\delta$. See Figure~\ref{F:triangle foliation}. 

Finally, recall that all geodesics in $G_i$ are either nonsingular and uniquely determined by any arc contained in it, or singular and contain exactly one cone point.  In the latter case, any arc $\beta$ in a singular geodesic $\tilde \gamma' \in G_i$ not containing the (unique) cone point of $\tilde \gamma'$ is contained in exactly one other geodesic $\tilde \gamma'' \in G_1$, cone-point asymptotic to $\tilde \gamma'$.  These two statements imply the claim about the fibers of $h$.
\end{proof}

With this basic structure in hand, the goal of this section is to prove that the map $g$ preserves the parallel structure described above.

\begin{proposition}[Parallelism is preserved] \label{P:split strips preserved}
Suppose $\tilde T$ is a triangle in $\widetilde{\mathcal T}$ and $\tilde \gamma \in G_1$ is a $\tilde \varphi_1$--generic geodesic meeting $\tilde T$.  Then $g(\tilde \gamma)$ is a $\tilde \varphi_2$--generic geodesic and
\[ \G^*(\tilde T,g(\tilde \gamma)) = g(\G^*(\tilde T,\tilde \gamma)).\]
\end{proposition}

We will need the following construction in the proof of this proposition.  Suppose $\tilde \gamma_0,\tilde \gamma_1 \in G_i$ are a pair of cone-point asymptotic geodesics.  For each $j=0,1$, we write $\tilde\gamma_j = \tilde \gamma_j^- \cup \tilde \gamma_j^+$ as a union of rays based at the (unique) cone point $\zeta$ in $\tilde \gamma_j$, 
so that $\tilde \gamma_0^+ = \tilde \gamma_1^+$.  The union of the negative rays $\tilde \gamma_0^- \cup \tilde \gamma_1^-$ bounds a closed {\em slit space} containing $\tilde \gamma_0^+ = \tilde \gamma_1^+$ which we denote $\mathcal S(\tilde \gamma_0,\tilde \gamma_1)$;  see Figure~\ref{F:V-half-plane}.   The boundary of $\mathcal S(\tilde \gamma_0,\tilde \gamma_1)$ is precisely $\tilde \gamma_0^- \cup \tilde \gamma_1^-$ and contains the single cone point $\zeta$, making an interior angle of $2 \pi$.

\begin{figure}[h]
\begin{center}
\begin{tikzpicture}[scale = .95]
\draw [draw=white, fill=gray, opacity=0.2] (-2,1) -- (2,0) -- (14,0) -- (14,3) -- (-2,3) -- (-2,1);
\draw [draw=white, fill=gray, opacity=0.2] (-2,-1) -- (2,0) -- (14,0) -- (14,-3) -- (-2,-3) -- (-2,-1);
\draw (-2,1) -- (2,0) -- (14,0);
\draw (-2,-1) -- (2,0);
\draw[fill=black] (2,0) circle (.08cm);
\draw [domain=-165:165] plot ({2+.2*cos(\x)}, {.2*sin(\x)});
\node at (2.2,.5) {$2 \pi$};
\node at (1.1,0) {\small $\zeta$};
\node at (1,.8) {$\tilde \gamma_0$};
\node at (1,-.9) {$\tilde \gamma_1$};
\node at (-1.2,.5) {$\tilde \gamma_0^-$};
\node at (-1.4,-.5) {$\tilde \gamma_1^-$};
\node at (9,.4) {$\tilde \gamma_0^+ = \tilde \gamma_1^+$};
\node at (13,1) {$\mathcal S(\tilde \gamma_0,\tilde \gamma_1)$};
\node at (6,2) {$H_0$};
\node at (6,-2) {$H_1$};
\end{tikzpicture}
\caption{The "slit space" $\mathcal S(\tilde \gamma_0,\tilde \gamma_1)$ is a union of half-planes $H_0$ and $H_1$ bounded by geodesics $\tilde \gamma_0$ and $\tilde \gamma_1$, respectively, along their maximal common sub-ray $\tilde \gamma_0^+ = \tilde \gamma_1^+$ based at the cone point $\zeta$.}
\label{F:V-half-plane}
\end{center}
\end{figure}
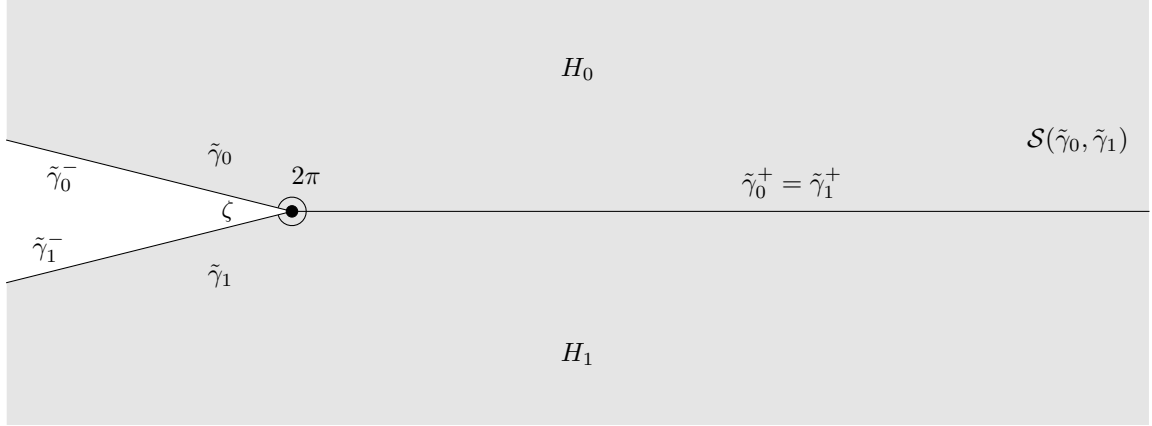

Next we observe that these slit spaces hit the cone points in the same way as their $g$--images.

\begin{lemma} \label{L:V-half-planes partition} If $\tilde \gamma_0,\tilde \gamma_1 \in G_1$ are cone-point asymptotic geodesics, then so are $g(\tilde \gamma_0),g(\tilde \gamma_1)$ and
\[ \mathcal S(\tilde \gamma_0,\tilde \gamma_1) \cap \tilde \Sigma = \mathcal S(g(\tilde \gamma_0),g(\tilde \gamma_1)) \cap \tilde \Sigma.\]
\end{lemma}
\begin{proof} As already noted $\mathcal S(\tilde \gamma_0,\tilde \gamma_1)$ is a union of closed half-planes $H_0$ and $H_1$ bounded by $\tilde \gamma_0$ and $\tilde \gamma_1$, respectively;  see Figure~\ref{F:V-half-plane}.  We have
\[ \mathcal S(\tilde \gamma_0,\tilde \gamma_1) \cap \tilde \Sigma =( H_0 \cap \tilde \Sigma) \cup ( H_1 \cap \tilde \Sigma).\]
Similarly, $\mathcal S(g(\tilde \gamma_0),g(\tilde \gamma_1)) = H_0' \cup H_1'$, where $H_0',H_1'$ are closed half-planes bounded by $g(\tilde \gamma_0)$ and $g(\tilde \gamma_1)$, respectively.  
Since cone-point partitions are preserved by $g$  (Lemma~\ref{geodcomb}), we have
\[ H_0' \cap \tilde \Sigma = H_0 \cap \tilde \Sigma \mbox{ and } H_1' \cap \tilde \Sigma = H_1 \cap \tilde \Sigma,\]
and the lemma follows.
\end{proof}

\begin{proof}[Proof of Proposition~\ref{P:split strips preserved}.]   From our  basic observations about how the foliation meets each triangle
(Lemma~\ref{L:generic triangle foliation}), there is a unique side $\delta$ of the triangle $\tilde T$ that meets all the geodesics in $\G^*(\tilde T,\tilde \gamma)$.  For all $\tilde \gamma' \in \G^*(\tilde T,\tilde \gamma)$, the two endpoints of $\delta$ (which are cone points) lie in different subsets of the partition of 
$\widetilde \Sigma$ determined by $\tilde \gamma'$, so since the partition is preserved (Lemma~\ref{geodcomb}), the same is true for $g(\tilde \gamma')$.  Since $\delta$ 
is also a $\tilde \varphi_2$--geodesic (by our convention identifying geodesics), this property of the partitions implies $g(\tilde \gamma') \cap \delta \neq \emptyset$.  

Let $\tilde \gamma_\alpha,\tilde \gamma_\omega$ denote the geodesics of $\G^*(\tilde T,\tilde \gamma)$ intersecting the endpoints of $\delta$ (i.e., the unique pair of geodesics in $\G^*(\tilde T,\tilde \gamma)$ that are disjoint from the interior of $\tilde T$) and set $\G^*_0(\tilde T,\tilde \gamma) = \G^*(\tilde T,\tilde \gamma) \setminus \{\tilde \gamma_\alpha,\tilde \gamma_\omega\}$.  Note that $g(\G^*_0(\tilde T,\tilde \gamma))$ consists of geodesics intersecting the interior of $\tilde T$ (and so also the interior of $\delta$).

\begin{claim} For all $\tilde \gamma' \in \G^*_0(\tilde T,\tilde \gamma)$, the images 
$g(\tilde \gamma)$ and $g(\tilde \gamma')$ intersect $\tilde T$ in parallel arcs.  
\end{claim}

Let us first establish that the proposition follows from this claim.  Observe that the claim implies $g(\G^*_0(\tilde T,\tilde \gamma)) \subset \G^*(\tilde T,g(\tilde \gamma))$, and since $g$ is continuous, we can extend this to both $\tilde \gamma_\alpha$ and $\tilde \gamma_\omega$, and therefore $g(\G^*(\tilde T,\tilde \gamma)) \subset \G^*(\tilde T,g(\tilde \gamma))$.
Now suppose $\tilde \gamma'' \in \G^*(\tilde T,g(\tilde \gamma))$ is any geodesic and consider its $\tilde \varphi_1$--straightening.  Observe that $\tilde \gamma''$ has no transverse intersections with any geodesic in $g(\G^*(\tilde T,\tilde \gamma))$ and lies between $g(\tilde \gamma_\alpha)$ and $g(\tilde \gamma_\omega)$.  Considering the endpoints on the circle at infinity, we see that a $\tilde \varphi_1$--straightening of $\tilde \gamma''$ lies between $\tilde \gamma_\alpha$ and $\tilde \gamma_\omega$ and has no transverse intersections with any geodesic from $\G^*(\tilde T,\tilde \gamma)$.  But such a $\tilde \varphi_1$--geodesic would have to intersect $\tilde T$ in a leaf of the foliation of $\tilde T$ coming from $\G^*(\tilde T,\tilde \gamma)$, and would thus be a geodesic in $\G^*(\tilde T,\tilde \gamma)$.  Therefore, $\tilde \gamma'' \in g(\G^*(\tilde T, \tilde \gamma))$ proving that $\G^*(\tilde T,g(\tilde \gamma)) \subset g(\G^*(\tilde T,\tilde \gamma))$, and hence $\G^*(\tilde T,g(\tilde \gamma)) = g(\G^*(\tilde T,\tilde \gamma))$, as required.  Therefore, all that remains is to prove the claim.

\begin{proof}[Proof of Claim.]  For the remainder of this proof, set $\mathfrak g = \G^*(\tilde T,\tilde \gamma)$ and $\mathfrak g_0 = \G^*_0(\tilde T,\tilde \gamma)$ to make the notation less cumbersome.   We start by defining
$Z_1$ to be the subsurface of $\tilde S$ foliated by geodesics from $\mathfrak g$.  It has singular $\tilde \varphi_1$--geodesic boundary.
The proof strategy will be to define a corresponding $Z_2$ and to show that it can be "zipped up" to a nonsingular Euclidean subsurface
with a well-defined notion of angle.  This will let us conclude that parallelism is suitably maintained.

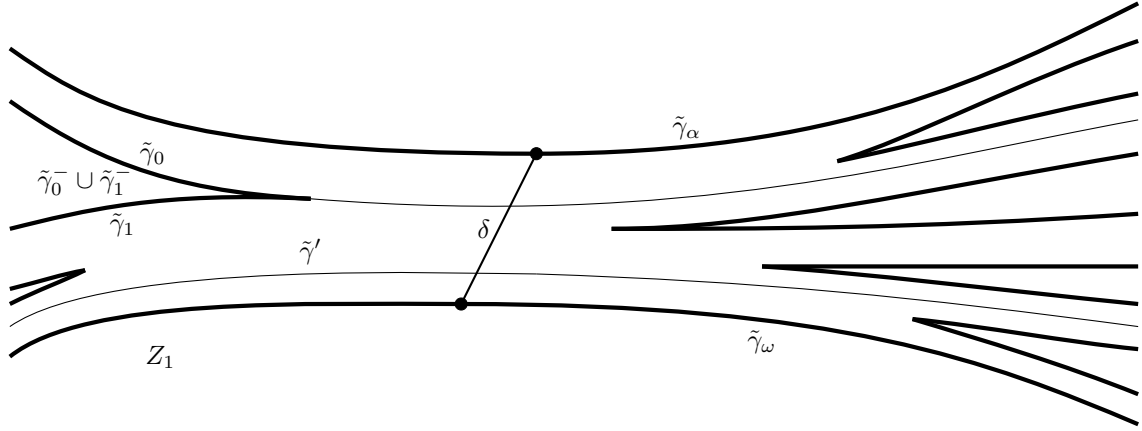
\begin{figure}[h]
\begin{center}
\begin{tikzpicture}[scale = 1]
\draw[ultra thick] (-2,3.4) .. controls (-1,2.7) and (0,2) .. (5,2); 
\draw[ultra thick] (5,2) .. controls (9,2) and (11,3) .. (13,4); 
\draw[ultra thick] (-2,-.7) .. controls (-1,.1) and (2,0) .. (4,0); 
\draw[ultra thick] (4,0) .. controls (9,0) and (11,-.75) .. (13,-1.6); 
\draw[ultra thick] (-2,2.7) .. controls (-1,2) and (0,1.5) .. (2,1.4); 
\draw[ultra thick] (-2,1) .. controls (-1,1.25) and (0,1.5) .. (2,1.4); 
\draw[ultra thick] (-2,.2) .. controls (-1.6,.3) and (-1.3,.4) .. (-1,.45); 
\draw[ultra thick] (-2,0) .. controls (-1.6,.2) and (-1.3,.3) .. (-1,.45); 
\draw[ultra thick] (6,1) .. controls (8,1) and (10,1.5) .. (13,2);  
\draw[ultra thick] (6,1) .. controls (8,1) and (10,1) .. (13,1.2); 
\draw[ultra thick] (9,1.9) .. controls (10,2.2) and (11.5,3) .. (13,3.5); 
\draw[ultra thick] (9,1.9) .. controls (10,2.1) and (11.5,2.5) .. (13,2.8); 
\draw[ultra thick] (8,.5) .. controls (9.5,.5) and (11,.5) .. (13,.5); 
\draw[ultra thick] (8,.5) .. controls (9.5,.4) and (11,.2) .. (13,0); 
\draw[ultra thick] (10,-.2) .. controls (11,-.3) and (12,-.5) .. (13,-.6); 
\draw[ultra thick] (10,-.2) .. controls (11,-.5) and (12.,-.8) .. (13,-1.2); 
\draw (2,1.4) .. controls (7,1) and (10,1.9) .. (13,2.45); 
\draw (-2,-.3) .. controls (-.7,.6) and (4,.4) .. (5,.4); 
\draw (5,.4) .. controls (8,.3) and (10,.1) .. (13,-.3); 
\draw[thick] (5,2) -- (4,0); 
\draw[fill=black] (5,2) circle (.08cm);
\draw[fill=black] (4,0) circle (.08cm);
\node at (7,2.35) {$\tilde \gamma_\alpha$};
\node at (8,-.5) {$\tilde \gamma_\omega$};
\node at (-.1,2) {$\tilde \gamma_0$};
\node at (-.5,1.05) {$\tilde \gamma_1$};
\node at (2,.7) {$\tilde \gamma'$};
\node at (-1,1.65) {$\tilde \gamma_0^- \cup \tilde \gamma_1^-$};
\node at (4.3,1) {$\delta$};
\node at (0,-.7) {$Z_1$};
\end{tikzpicture}
\caption{The surface $Z_1$ together with the geodesic segment $\delta$ intersected by all geodesics in $\mathfrak g$, the nonsingular boundary components $\tilde \gamma_\alpha$ and $\tilde \gamma_\omega$, a pair $\{\tilde \gamma_0,\tilde \gamma_1\} \in P\mathfrak g_0$ defining a singular boundary component $\tilde \gamma_0^- \cup \tilde \gamma_1^-$, and another generic geodesic $\tilde \gamma' \in \mathfrak g_0$.}
\label{F:slit strip}
\end{center}
\end{figure}

We will need an alternative description of $Z_1$ as follows.  From Lemma~\ref{L:generic triangle foliation} on the structure of the geodesics in $\mathfrak g$, all singular geodesics in $\mathfrak g_0$ occur in cone-point asymptotic pairs, and we consider the set of all such pairs:
\[ P\mathfrak g_0 = \{ \{\tilde \gamma_0,\tilde \gamma_1 \} \mid \tilde \gamma_0,\tilde \gamma_1 \in \mathfrak g_0, \mbox{ cone-point asymptotic } \},\]
which has an obvious map to cone points.
Let $\mathcal H(\tilde \gamma_\alpha)$ and $\mathcal H(\tilde \gamma_\omega)$ denote the half-spaces bounded by $\tilde \gamma_\alpha$ and $\tilde \gamma_\omega$, respectively, containing $\delta$.  Then, we have
\[ Z_1 = \mathcal H(\tilde \gamma_\alpha) \cap \mathcal H(\tilde \gamma_\omega)\cap \bigcap_{\{\tilde \gamma_0,\tilde \gamma_1 \} \in P\mathfrak g_0} \mathcal S(\tilde \gamma_0,\tilde \gamma_1),\]
where $\mathcal S(\tilde \gamma_0,\tilde \gamma_1)$ is the slit space bounded by the union of subrays $\tilde \gamma_0^-$ and $\tilde \gamma_1^-$ of $\tilde\gamma_0$ and $\tilde\gamma_1$, respectively.
That $Z_1$ is a subsurface follows from this description since the boundaries of the half planes and slit spaces in the intersection form a locally finite set.
The boundary of $Z_1$ decomposes as the union of the boundaries of these half-planes and slit spaces, which are precisely $\tilde \gamma_\alpha$, $\tilde \gamma_\omega$, and the union of pairs of rays $\tilde \gamma_0^- \cup \tilde \gamma_1^-$, one for each $\{\tilde \gamma_0,\tilde \gamma_1 \} \in P\mathfrak g_0$.  Each boundary component of the form $\tilde \gamma_0^- \cup \tilde \gamma_1^-$ contains exactly one cone point and makes cone angle $2 \pi$ on the interior.  There are no cone points in the interior of $Z_1$.  Although there are cone points (the endpoints of $\delta$) on the two boundary components, $\tilde \gamma_\alpha$ and $\tilde \gamma_\omega$, the interior cone angle in $Z_1$ is $\pi$.  See Figure~\ref{F:slit strip} for an illustration of the various features of $Z_1$.

From this description, we obtain a  subsurface of $\tilde S$ for the $\tilde \varphi_2$--metric in an exactly similar way:
\[ Z_2 = \mathcal H(g(\tilde \gamma_\alpha)) \cap \mathcal H(g(\tilde \gamma_\omega)) \cap \bigcap_{\{\tilde \gamma_0,\tilde \gamma_1 \} \in P\mathfrak g_0} \mathcal S(g(\tilde \gamma_0),g(\tilde \gamma_1)).\]

The boundary of $Z_2$ decomposes just like $Z_1$.  By Lemma~\ref{geodcomb} and Lemma~\ref{L:V-half-planes partition}, $Z_2$ has no cone points in the interior and exactly one cone point on each boundary component.  For each pair $\{\tilde \gamma_0,\tilde \gamma_1\} \in P\mathfrak g_0$, the corresponding boundary component of $Z_2$ is of the form
\[ g(\tilde \gamma_0)^- \cup g(\tilde \gamma_1)^-\]
and has interior cone angle $2\pi$ at the unique cone point  it contains, while $g(\tilde \gamma_\alpha)$ and $g(\tilde \gamma_\omega)$ have interior cone angle $\pi$ at their cone point.

We now construct a quotient {\em nonsingular} Euclidean surface $\Pi \colon Z_2 \to \hat Z_2$ by "zipping up the slits":  
we isometrically identify the pair of rays $g(\tilde \gamma_0)^-$ and $g(\tilde \gamma_1)^-$, for each 
$\{\tilde \gamma_0,\tilde \gamma_1\} \in P\mathfrak g_0$.  Away from the cone points, $\Pi$ is a local isometric embedding, and it is 
a local isometry away from $\partial Z_2$.  In particular, $\Pi(\delta)$ is a geodesic arc connecting the boundary components $\Pi(g(\tilde \gamma_\alpha))$ and $\Pi(g(\tilde \gamma_\omega))$, and for any geodesic $\tilde \gamma' \in \mathfrak g_0$, $\Pi(g(\tilde \gamma'))$ is a geodesic in $\hat Z_2$ intersecting $\Pi(\delta)$ in the same angle that $g(\tilde \gamma')$ intersects $\delta$.  Observe that for any pair $\{ \tilde \gamma_0,\tilde \gamma_1 \} \in P\mathfrak g_0$ we have $\Pi(g(\tilde \gamma_0)) = \Pi(g(\tilde \gamma_1))$; see Figure~\ref{F:slit strip to nonsingular}.

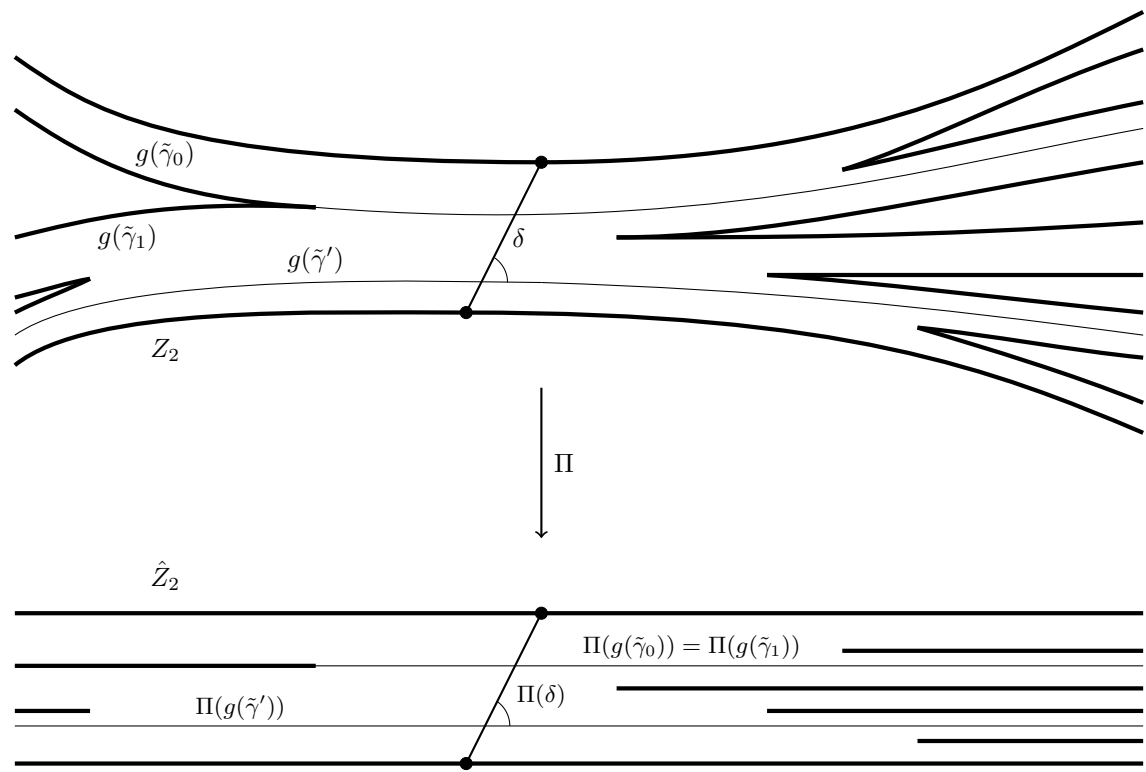
\begin{figure}[h]
\begin{center}
\begin{tikzpicture}[scale = 1]
\draw[ultra thick] (-2,3.4) .. controls (-1,2.7) and (0,2) .. (5,2); 
\draw[ultra thick] (5,2) .. controls (9,2) and (11,3) .. (13,4); 
\draw[ultra thick] (-2,-.7) .. controls (-1,.1) and (2,0) .. (4,0); 
\draw[ultra thick] (4,0) .. controls (9,0) and (11,-.75) .. (13,-1.6); 
\draw[ultra thick] (-2,2.7) .. controls (-1,2) and (0,1.5) .. (2,1.4); 
\draw[ultra thick] (-2,1) .. controls (-1,1.25) and (0,1.5) .. (2,1.4); 
\draw[ultra thick] (-2,.2) .. controls (-1.6,.3) and (-1.3,.4) .. (-1,.45); 
\draw[ultra thick] (-2,0) .. controls (-1.6,.2) and (-1.3,.3) .. (-1,.45); 
\draw[ultra thick] (6,1) .. controls (8,1) and (10,1.5) .. (13,2);  
\draw[ultra thick] (6,1) .. controls (8,1) and (10,1) .. (13,1.2); 
\draw[ultra thick] (9,1.9) .. controls (10,2.2) and (11.5,3) .. (13,3.5); 
\draw[ultra thick] (9,1.9) .. controls (10,2.1) and (11.5,2.5) .. (13,2.8); 
\draw[ultra thick] (8,.5) .. controls (9.5,.5) and (11,.5) .. (13,.5); 
\draw[ultra thick] (8,.5) .. controls (9.5,.4) and (11,.2) .. (13,0); 
\draw[ultra thick] (10,-.2) .. controls (11,-.3) and (12,-.5) .. (13,-.6); 
\draw[ultra thick] (10,-.2) .. controls (11,-.5) and (12.,-.8) .. (13,-1.2); 
\draw (2,1.4) .. controls (7,1) and (10,1.9) .. (13,2.45); 
\draw (-2,-.3) .. controls (-.7,.6) and (4,.4) .. (5,.4); 
\draw (5,.4) .. controls (8,.3) and (10,.1) .. (13,-.3); 
\draw[->,thick] (5,-1) -- (5,-3);
\draw[ultra thick] (-2,-4) -- (13,-4);
\draw (-2,-4.7) -- (13,-4.7);
\draw[ultra thick] (-2,-6) -- (13,-6);
\draw[ultra thick] (-2,-4.7) -- (2,-4.7);
\draw[ultra thick] (-2,-5.3) -- (-1,-5.3);
\draw[ultra thick] (6,-5) -- (13,-5);
\draw[ultra thick] (9,-4.5) -- (13,-4.5);
\draw[ultra thick] (8,-5.3) -- (13,-5.3);
\draw[ultra thick] (10,-5.7) -- (13,-5.7);
\draw (-2,-5.5) -- (13,-5.5);
\node at (0,-3.5) {$\hat Z_2$};
\node at (0,-.5) {$Z_2$};
\node at (5.3,-2) {$\Pi$};
\draw[thick] (5,2) -- (4,0); 
\draw[fill=black] (5,2) circle (.08cm);
\draw[fill=black] (4,0) circle (.08cm);
\draw[thick] (5,-4) -- (4,-6); 
\draw[fill=black] (5,-4) circle (.08cm);
\draw[fill=black] (4,-6) circle (.08cm);
\node at (4.7,1) {$\delta$};
\node at (0,2.1) {$g(\tilde \gamma_0)$};
\node at (-.5,1) {$g(\tilde \gamma_1)$};
\node at (2,.7) {$g(\tilde \gamma')$};
\node at (7,-4.45) {\small $\Pi(g(\tilde \gamma_0)) = \Pi(g(\tilde \gamma_1))$};
\node at (1,-5.25) {\small $\Pi(g(\tilde \gamma'))$};
\node at (5,-5.1) {\small $\Pi(\delta)$};
\draw [domain=0:55] plot ({4.15+.4*cos(\x)}, {.4+.4*sin(\x)});
\draw [domain=0:55] plot ({4.18+.4*cos(\x)}, {-5.5+.4*sin(\x)});
\end{tikzpicture}
\caption{The surface $Z_2$ and nonsingular quotient $\Pi \colon Z_2 \to \hat Z_2$, together with: $g(\tilde \gamma') \in g(\mathfrak g_0)$, the cone-point asymptotic pair $\{g(\tilde \gamma_0),g(\tilde \gamma_1)\} \in g(P\mathfrak g_0)$, and $\delta$, as well as their images under $\Pi$. The angle between $\delta$ and $g(\tilde \gamma')$ is equal to the angle between $\Pi(\delta)$ and $\Pi(g(\tilde \gamma'))$.}
\label{F:slit strip to nonsingular}
\end{center}
\end{figure}

Now let $\tilde \gamma' \in \mathfrak g_0$ be any disjoint geodesic distinct from $\tilde \gamma$, and suppose that $g(\tilde \gamma')$ and $g(\tilde \gamma)$ intersect $\delta$ in different angles.  The same is then true of the images in $\hat Z_2$.  Since $\hat Z_2$ is simply connected and complete with geodesic boundary, the developing map is globally defined and injective, and is thus an isometric embedding to a closed subset of $\mathbb R^2$.  Therefore, the images of $\Pi(g(\tilde \gamma'))$ and $\Pi(g(\tilde \gamma))$, being bi-infinite lines in $\mathbb R^2$ making different angles with the image of $\Pi(\delta)$ must intersect.  On the other hand, $\tilde \gamma'$ and $\tilde \gamma$ are disjoint, and hence so are $g(\tilde \gamma')$ and $g(\tilde \gamma)$.  By construction of $\hat Z_2$, $\Pi(g(\tilde \gamma'))$ and $\Pi(g(\tilde \gamma))$ are also disjoint, and hence so are their images in $\mathbb R^2$.  This is impossible, and so $g(\tilde \gamma')$ and $g(\tilde \gamma)$ intersect $\delta$ in the same angle, proving the claim.
\end{proof}
Since the claim proves the proposition, we are done.
\end{proof}

\subsection{Holonomy preserved} \label{S:holonomy}

Our next goal is to use this preservation of parallelism to conclude that holonomy is also preserved.

Suppose that $\tilde T$ is a triangle in $\widetilde{\mathcal T}$ and $\tilde \gamma \in G_i$ is a $\tilde \varphi_i$--generic geodesic for $\tilde T$ (and so intersecting the interior of $\tilde T$).  Observe that any other geodesic in $\G^*(\tilde T,\tilde \gamma)$ intersecting the interior of $\tilde T$ is also $\tilde \varphi_i$--generic for $\tilde T$.   Therefore, if $\theta \in S^1_i(\tilde T)$ is the direction in $\tilde T$ of $\tilde \gamma$, we can unambiguously refer to $\theta$ as a {\em $\tilde \varphi_i$--generic direction}, and we write $\G^*(\tilde T,\theta) = \G^*(\tilde T,\tilde \gamma)$.  If $T$ is the image $T = p(\tilde T)$ in $S$, then we have an identification of the space of directions $S^1_i(\tilde T) = S^1_i(T)$, and we define a direction in $S^1_i(T)$ to be $\varphi_i$--generic if it is $\tilde \varphi_i$--generic in $S^1_i(\tilde T)$.  Since $G_1$ is invariant under the action of $\pi_1S$, this is independent of the choice of triangle $\tilde T$ in the preimage of $T$.  Note that for any triangle $T$, there are only countably many non-$\varphi_i$--generic directions.

Let $\Gamma$ denote the dual graph of the $1$--skeleton $\mathcal T^{(1)}$, which we view as embedded transversely and minimally intersecting $\mathcal T^{(1)}$.  An edge path $\delta \colon [0,1] \to \Gamma$ has initial vertex $\delta(0)$ in some unique triangle $T$ of $\mathcal T$ and terminal vertex in a unique triangle $T'$ of $\mathcal T$.  We say that $\delta$ is an {\em edge path of $\Gamma$ from $T$ to $T'$}, and we denote the initial and terminal vertices as $v_T = \delta(0)$ and $v_{T'} = \delta(1)$.  Given such a path $\delta$ and direction $\theta \in S^1_i(T)$, we can $\varphi_i$--parallel translate $\theta$ along $\delta$ to obtain a $\varphi_i$--direction $P_{\delta,\varphi_i}(\theta)$ in $T'$.  The parallel translate  $P_{\delta,\varphi_i}(\theta)$ depends only on the homotopy class of $\delta$, rel endpoints, in $\Gamma$ (or equivalently in $S \setminus \Sigma)$ since the metric is Euclidean away from the cone points.  In particular, if $v_T \in \Gamma^{(0)}$ is the base point in $T$ then this defines the $\varphi_i$--holonomy homomorphism
\[ P_{\varphi_i} \colon \pi_1(\Gamma,v_T) \cong \pi_1(S \setminus \Sigma,v_T) \to SO(2),\]
by $P_{\varphi_i}([\delta]) \cdot \theta = P_{\delta,\varphi_i}(\theta)$ for any loop $\delta$ based at $v_T$.

Given a triangle $T$ of $\mathcal T$, we say that a direction $\theta \in S^1_i(T)$ is {\em $\varphi_i$--stably generic} if for every triangle $T'$ in $\mathcal T$ and path $\delta$ from $T$ to $T'$, $P_{\delta,\varphi_i}(\theta) \in S^1_i(T')$ is $\varphi_i$--generic.   It follows that for any path $\delta$ in $\Gamma$ from $T$ to $T'$, the direction $\theta \in S^1_i(T)$ is $\varphi_i$--stably generic if and only if $P_{\delta,\varphi_i}(\theta) \in S^1_i(T')$ is $\varphi_i$--stably generic.  Since there are only countably many homotopy classes of paths between any two vertices of $\Gamma$, and only countably many non-$\varphi_i$--generic directions in $S^1_i(T)$, for any $T$ in $\mathcal T$, it follows that there are only countably many non-$\varphi_i$--stably generic directions in $S^1_i(T)$.  We denote the complementary set of $\varphi_i$--stably generic directions in $T$ by
\[ \Delta(T,\varphi_i) = \{ \theta \in S^1_i(T) \mid \theta \mbox{ is $\varphi_i$--stably generic in $T$ } \}.\]

The next result is the key to determining angles of triangles when the holonomy is infinite.  
\begin{proposition}[Defining angles] \label{P:holonomy conjugating homeo} For every triangle $T\in \mathcal T$ there is an orientation-preserving homeomorphism
\[ F_T \colon S^1_1(T) \to S^1_2(T)\]
which is $\pi_1(\Gamma,v_T)$--equivariant with respect to the holonomy homomorphisms:
\[ F_T(P_{\varphi_1}(\delta) \cdot \theta) = P_{\varphi_2}(\delta) \cdot F_T(\theta).\]
Furthermore, if $\tilde T$ is a triangle in $\widetilde{\mathcal T}$ with $p(\tilde T) = T$ and $\theta \in \Delta(T,\varphi_1)$, then $F_T(\theta) \in \Delta(T,\varphi_2)$ and
\[ g(\G^*(\tilde T,\theta)) = \G^*(\tilde T,F_T(\theta)).\]
\end{proposition}
By the last statement, $F_T$ is determined by $g$ via the stably generic geodesic foliations of $\tilde T$.
\begin{proof}  Fix any triangle $\tilde T$ in $\widetilde{\mathcal T}$.  According to Proposition~\ref{P:split strips preserved}, for any $\tilde \varphi_1$--generic direction $\theta \in S^1_1(\tilde T)$ and $\tilde \varphi_1$--geodesic $\tilde \gamma$ intersecting $\tilde T$ in direction $\theta$, we have that $g(\tilde \gamma)$ is a $\tilde \varphi_2$--generic geodesic, and we can uniquely define $F_{\tilde T}(\theta) \in S^1_2(\tilde T)$ to be the $\tilde \varphi_2$--generic direction of $g(\tilde \gamma)$.  This uniquely determines a bijection $F_{\tilde T}$ from the set of  $\tilde \varphi_1$--generic directions in $S^1_1(\tilde T)$ to the $\tilde \varphi_2$--generic directions in $S^1_2(\tilde T)$ satisfying
\[ g(\G^*(\tilde T,\theta)) = g(\G^*(\tilde T,\tilde \gamma)) =  \G^*(\tilde T,g(\tilde \gamma))  = \G^*(\tilde T,F_{\tilde T}(\theta)).\]

Consider any three directions $\theta_1,\theta_2,\theta_3$ appearing cyclically in this order around $S^1_1(\tilde T)$ and let $\tilde \gamma_1,\tilde \gamma_2,\tilde \gamma_3$  be $\tilde \varphi_1$--geodesics through a single point $x$ in the interior of $\tilde T$ in each of these directions, respectively.  The forward endpoints on the circle at infinity for each of these geodesics is the same as for their images by $g$, and hence $F_{\tilde T}(\theta_1),F_{\tilde T}(\theta_2),F_{\tilde T}(\theta_3)$ also appear cyclically in the same order around $S^1_2(\tilde T)$.  That is, $F_{\tilde T}$ preserves the cyclic ordering.  It follows that for $T = p(\tilde T)$, $F_{\tilde T}$ uniquely extends to a homeomorphism $F_T \colon S^1_1(T) \to  S^1_2(T)$.
By construction, the last condition of the proposition is satisfied.

Next, suppose $\tilde T, \tilde T'$ are triangles that share an edge $\tilde e$, let $T = p(\tilde T)$ and $T' = p(\tilde T')$ with $e = p(\tilde e) \in T \cap T'$ the image of the shared edge, and let $\delta$ be the edge of $\Gamma$ dual to $e$ with $\tilde \delta$ a lift of $\delta$ intersecting $\tilde e$.
Suppose $\theta \in S^1_1(T)$ is any $\varphi_1$--stably generic direction and $\tilde \gamma$ is a geodesic intersecting $\tilde T$ in the direction $\theta$, and suppose that $\tilde \gamma$ crosses $\tilde e$.  From the property of $F_T$ we have already proved, we know that the direction of $g(\tilde \gamma)$ in $\tilde T$ is $F_T(\theta)$.  Since $\tilde \gamma$ and $g(\tilde \gamma)$ cross $\tilde e$, both intersect $\tilde T'$.  Parallel transport along an arc of $\tilde \gamma$ contained in the union $\tilde T \cup \tilde T'$ or along $\tilde \delta$ (or equivalently along $\delta$ via the identification from the covering $p$) both define the same maps $P_{\delta,\varphi_1} \colon S^1_1(T) \to S^1_1(T')$, and it follows that the direction of $\tilde \gamma$ in $\tilde T'$ is $P_{\delta,\varphi_1}(\theta)$.  Similarly, the direction of $g(\tilde \gamma)$ in $\tilde T'$ is $P_{\delta,\varphi_2}(F_T(\theta))$.  On the other hand, applying what we already proved about $F_{T'}$, we know that the direction of $g(\tilde \gamma)$ in $\tilde T'$ is $F_{T'}(P_{\delta,\varphi_1}(\theta))$.  That is,
\begin{equation} \label{E:holonomy eq 1} F_{T'} \circ P_{\delta,\varphi_1}(\theta) = P_{\delta,\varphi_2} \circ F_T(\theta).
\end{equation}
Since this is true on a dense set in $S^1_1(T)$, it is true on all of $S^1_1(T)$, by continuity.

If $\delta = \delta_1 \delta_2 \cdots \delta_k$ is any edge path connecting vertices $v_{T_0},v_{T_1}, \cdots , v_{T_k}$ (so that $\delta_j$ is an edge from $v_{T_{j-1}}$ to $v_{T_j}$), then by repeatedly applying Equation (\ref{E:holonomy eq 1}) we have
\[ P_{\delta,\varphi_2} \circ F_{T_0} = P_{\delta_2 \cdots \delta_k,\varphi_2} \circ F_{T_1} \circ P_{\delta_1,\varphi_1} = P_{\delta_3 \cdots \delta_k,\varphi_2} \circ F_{T_2} \circ P_{\delta_1 \delta_2,\varphi_1} = \cdots = F_{T_k} \circ P_{\delta,\varphi_1}. \]
The fact that $F_T$ is equivariant with respect to the holonomy homomorphisms $P_{\varphi_1}$ and $P_{\varphi_2}$ now follows by taking $\delta$ in this equation to be a loop based at a vertex of $\Gamma$.
\end{proof}

\begin{corollary} \label{C:equal holonomy} $\varphi_1$ and $\varphi_2$ induce the same holonomy homomorphism
\[ P_{\varphi_1} = P_{\varphi_2} \colon \pi_1(S \setminus \Sigma) \to SO(2).\]
\end{corollary}
\begin{proof} For every $\gamma \in \pi_1(\Gamma,v_T) = \pi_1(S \setminus \Sigma,v_T)$, Proposition~\ref{P:holonomy conjugating homeo} implies that $F_T$ is an orientation-preserving topological conjugacy between $P_{\varphi_1}(\gamma)$ and $P_{\varphi_2}(\gamma)$.   Any rotation is determined by the cyclic ordering on any orbit and its action on this cyclically ordered set.  Consequently, $P_{\varphi_1}(\gamma) = P_{\varphi_2}(\gamma)$, and since $\gamma$ was arbitrary, $P_{\varphi_1} = P_{\varphi_2}$.
\end{proof}

When the holonomy is infinite, we get the most information about the homeomorphism $F_T$.

\begin{corollary}~\label{C:infinite holonomy isometry} If $P_{\varphi_1}(\pi_1(\Gamma))$ is infinite, then for each triangle $T$, the map 
$F_T$ is an isometric identification of the circle of directions in the two metrics. 
In particular, $F_T$ conjugates $P_{\varphi_1}$ to $P_{\varphi_2}$ {\em inside} $SO(2)$.
\end{corollary}
To emphasize this last point, we note that $F_T$ is a priori a topological conjugacy, but when the holonomy is infinite, it is in fact an {\em isometric} conjugacy.
\begin{proof} Since $P_{\varphi_1}(\pi_1(\Gamma,v_T))$ is an infinite, finitely generated abelian group, there exists an element $[\delta] \in \pi_1(\Gamma,v_T)$ such that $P_{\varphi_1}([\delta])$ has infinite order.  Then both $P_{\varphi_1}([\delta])$ and $P_{\varphi_2}([\delta])$ are irrational rotations, topologically conjugate by the homeomorphism $F_T$.  It follows that their angles of rotation are equal, and $F_T$ is an isometry on the orbit of any point by $\langle P_{\varphi_1}([\delta]) \rangle$.  Since this orbit is dense, $F_T$ is an isometry.
\end{proof}

From this, we easily deduce the following, which essentially proves the Support Rigidity Theorem when the holonomy has infinite order.

\begin{corollary} \label{C:infinite holonomy proof} If $P_{\varphi_1}(\pi_1(\Gamma))$ is infinite, then the $\varphi_1$--interior angle 
at any vertex of any triangle $T$ of $\mathcal T$ is equal to the $\varphi_2$--interior angle of the same vertex.
\end{corollary}
\begin{proof}  Let $\tilde T$ be a triangle with $p(\tilde T) = T$ and  for each $i = 1,2$, 
let $\Theta_i(\tilde T) \subset S^1_i(\tilde T)$ be the set of six $\tilde \varphi_i$--directions parallel to the sides of $\tilde T$ (each side is considered with both orientations so appears twice). 
For every $\theta \in S^1_i(\tilde T) \setminus \Theta_i(\tilde T)$, the foliation of $\tilde T$ determined by $\G^*(\tilde T,\theta)$  
is transverse to each of the sides.  Furthermore, for any such $\theta$, the side which meets every leaf defines a locally constant function
\[ h_i \colon S^1_i(\tilde T) \setminus \Theta_i(\tilde T) \to \mbox{sides}(\tilde T).\]
This function changes values at precisely each of the six directions in $\Theta_i(\tilde T)$.  In fact, $h_i$ changes value from one side to another at the direction of the third side.  Now note that for a $\tilde \varphi_1$--generic direction $\theta$, 
we have $h_1(\theta) = h_2(F_T(\theta))$, by Proposition~\ref{P:holonomy conjugating homeo}.  Consequently,
\[ F_T(\Theta_1(\tilde T)) = \Theta_2(\tilde T).\]
Therefore, since the previous result shows that $F_T$ is an isometry, the $\tilde \varphi_1$--directions of the three sides of $\tilde T$ differ from the $\tilde \varphi_2$--directions by an isometry, and consequently, the $\tilde \varphi_1$--angles and $\tilde \varphi_2$--angles of $\tilde T$ agree.  Pushing back down to $T$ proves the corollary.
\end{proof}

\subsection{Proof of the Support Rigidity Theorem} \label{S:support-rigidity}

We are now ready to put together all the pieces and show that a flat metric is determined (up to affine equivalence) by its support.

\begin{proof}[Proof of Support Rigidity Theorem]  We continue to assume, without loss of generality, that $\varphi_1$ and $\varphi_2$ are representatives chosen so that we can identify geodesics and triangulations as above.   
The proof of the Support Rigidity Theorem divides into two cases, depending on whether the holonomy is finite or infinite.

\medskip

\noindent
{\bf Case 1.} The holonomy of $\varphi_1$ is infinite.\\
 
 For every triangle $T$ of $\mathcal T$ and every vertex of $T$, the interior 
 $\varphi_1$--angle is equal to the interior $\varphi_2$--angle  (Corollary~\ref{C:infinite holonomy proof}), and thus, the $\varphi_1$--metric and $\varphi_2$--metric on $T$ are similar.  
 Since the identity $(S,\varphi_1) \to (S,\varphi_2)$ is already affine on each triangle $T \in \mathcal T$, it is in fact a similarity, and so it scales distances by a constant factor $c_T$.  When two triangles share an edge, their scaling factors must be equal.  Since $S$ is connected, $c_T = c_{T'}$ for any two triangles $T,T'$ of $\mathcal T$, and hence $\varphi_1$ and $\varphi_2$ differ by some global scalar $c >0$.  
 Since both $\varphi_1$ and $\varphi_2$ have unit area, the scalar $c$ must equal $1$, and hence $\varphi_1 = \varphi_2$.\\

\noindent
{\bf Case 2.} The holonomy of $\varphi_1$ is finite.\\

Since the two metrics induce the same holonomy homomorphism (Corollary~\ref{C:equal holonomy}), 
we can let $\pi \colon S' \to S$ denote the branched cover corresponding to the kernel of this homomorphism $P_{\varphi_i}$.
More precisely, $S'$ is the metric completion of the cover of $S \setminus \Sigma$ corresponding to the kernel, and $\pi$ is the extension of the covering map to the completion (which sends completion points $\Sigma'$ of $S'$ to appropriate points of $\Sigma$).  Write $\varphi_i' = \pi^*(\varphi_i)$ for the pullback of $\varphi_i$ to $S'$, for $i=1,2$.  Since the cover was constructed from the kernel, 
both of these metrics have trivial holonomy.

Moreover, the cylinder sets are equal in the two metrics (Corollary~\ref{C:common cylinders}),
and the homotopy from a cylinder curve for $\varphi_1$ to a cylinder curve for $\varphi_2$ occurs is in the complement of $\Sigma$.  Now observe that the cylinders for $\varphi_i'$ are precisely the preimages of the cylinders for $\varphi_i$ under 
$\pi$, for $i=1,2$, and hence $\cyl(\varphi_1') = \cyl(\varphi_2')$ on $S' \setminus \Sigma'$.  
By passing to these covers we are now in $\Flat_1$ and we may invoke Theorem~\ref{DLR} to conclude from the equality of cylinder sets that 
$\varphi_1'$ and $\varphi_2'$ are affine-equivalent on $S' \setminus \Sigma'$.  Let $f'$ be the affine map, isotopic to the identity on $S' \setminus \Sigma'$, which extends by the identity over $\Sigma'$.

Now every triangle $T\in\mathcal T$ is covered by a triangle $T' \subset \pi^{-1}(T)$, and the map 
\[ f'|_{T'} \colon (T',\varphi_1') \to (f'(T'),\varphi_2')\]
is  affine. By our convention for triangulations, the identity $(T',\varphi_1') \to (T',\varphi_2')$ is also affine.  So, since $f'$ is isotopic to the identity rel $\Sigma'$, it follows that for each edge $e$ of $T'$, $f'(e)$ is the unique $\varphi_2'$--geodesic representative of the isotopy class (namely, the straight segment).   
Since $e$ is already a $\varphi_2'$--geodesic, these must be equal.  That is, $f'(T') = T'$.  Since both the identity and $f'$ are affine maps from $(T',\varphi_1')$ to $(T',\varphi_2')$, they must be equal.  Since $T'$ was an arbitrary triangle in the preimage of an arbitrary triangle of $\mathcal T$, it follows that $f'$ is the identity.  Therefore the identity $(S,\varphi_1) \to (S,\varphi_2)$ is affine, and hence $\varphi_1$ and $\varphi_2$ are affine equivalent.

If the holonomy has order greater than $2$, then affine-equivalent flat metrics are actually equal
(Proposition~\ref{P:affine deformation iff quadratic differential}), so  $\varphi_1=\varphi_2$.  This completes the proof.
\end{proof}



\section{Bounce Theorem and related results}\label{bounce}

We are now ready to prove our main theorem on billiards.  We begin by recalling the statement.

\begin{bouncetheorem} If finite sided, simply connected Euclidean polygons $P_1,P_2$ have 
$\B(P_1)=\B(P_2)$, then either $P_1,P_2$ are right-angled and affinely equivalent, or they are 
similar polygons.
\end{bouncetheorem}

\noindent
In \S\ref{sec:diag}--\ref{S:cutting seq} we describe corollaries of this result, and other applications of the ideas in the proof.

\subsection{Proof of the Bounce Theorem}

We view billiard trajectories in $P$ as unit speed piecewise geodesic paths $\tau:\mathbb{R}\to P$. 
Recall also that with respect to a cyclic labeling of the edges from 
an ordered alphabet $\mathcal{A}$,  the bounce spectrum $\B(P)$ is the set of all bounce sequences (sequences of letters from $\mathcal A$) 
that can occur along billiard trajectories.

In order to prove the Bounce Theorem, we will heavily use the results of the last section. To do this, 
we relate billiard trajectories to nonsingular geodesics on flat surfaces. 

\begin{definition}[Unfolding]
\label{unfolding}
Suppose $P$ is a simply connected Euclidean polygon with $n$ vertices, $X$ is an oriented surface of negative Euler characteristic,
and $G$ is a finite group acting faithfully on  $X$ with quotient $P$. Then the quotient map  $r: X \rightarrow P$ is called a \emph{folding map}. We pull back the flat metric on $P$ by $r$ to a metric on $X$. If the preimage of every vertex in $P$ is a cone point with angle 
more than $2\pi$, then $X$ together with its pullback metric is called an \emph{unfolding} of $P$.

We also give an equivalent constructive definition of an unfolding of $P$ and introduce some notation. This highlights that unfoldings are tessellated by lifts of $P$ and that each tile is a fundamental domain for the action of $G$. For the following description, 
 refer to Figure~\ref{F:unfolding} throughout.

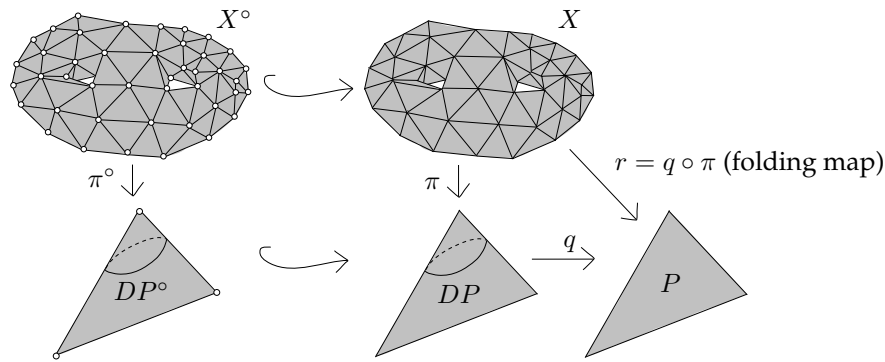
\begin{figure}[h]
\centering
\begin{tikzpicture}
\node (img) at (0,0) {\includegraphics[width=4in]{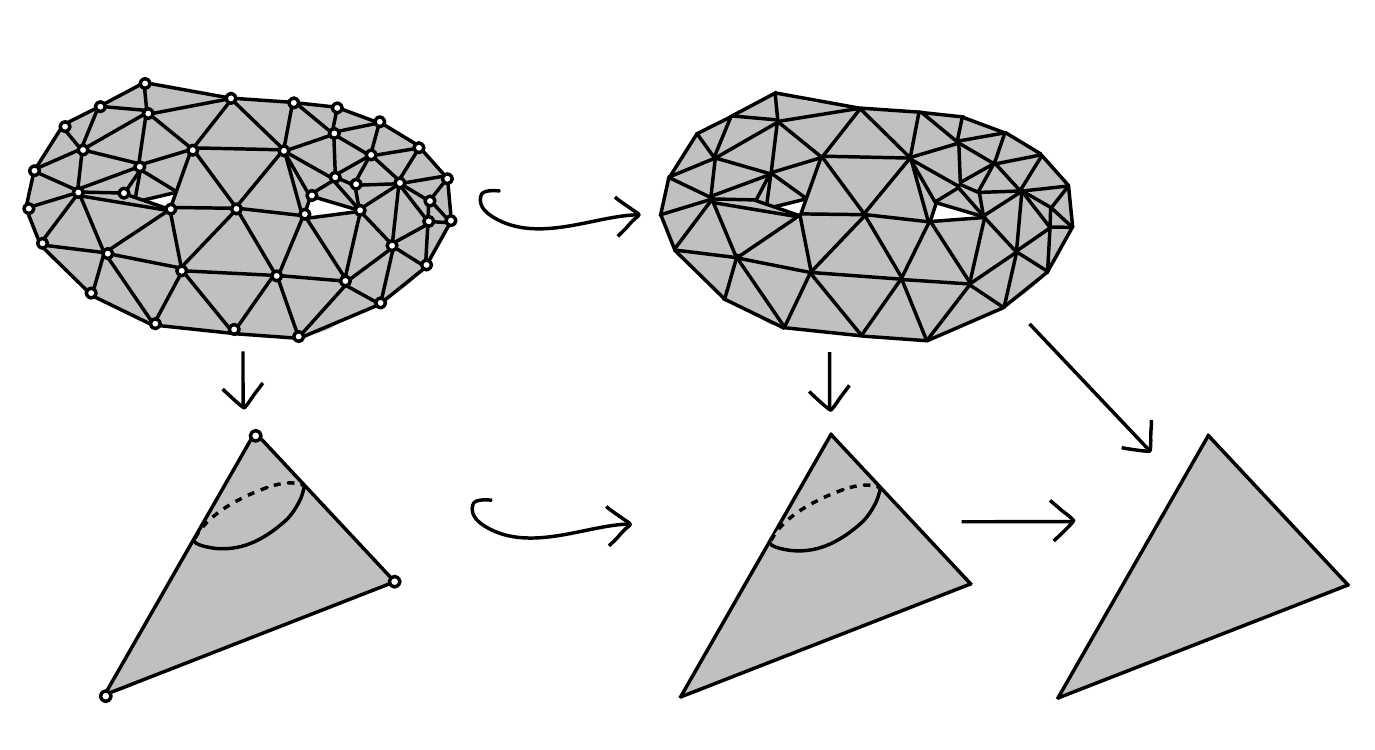}};
\node at (72pt,-25pt) {$q$};
\node at (72pt,60pt) {$X$};
\node at (-55pt,60pt) {$X^{\circ}$};
\node at (-105pt,0pt) {$\pi^{\circ}$};
\node at (20pt,-2pt) {$\pi$};
\node at (140pt,5pt) {$r=q \circ \pi$ (folding map)};
\node at (-90pt,-43pt) {$DP^{\circ}$};
\node at (30pt,-44pt) {$DP$};
\node at (110pt,-40pt) {$P$};
\end{tikzpicture}
\caption{A genus--$2$ unfolding of $P$}
\label{F:unfolding}
\end{figure}

Let $DP$ be the double of $P$, obtained by gluing two copies of $P$ along corresponding edges. There is an involution interchanging the two copies of $P$ and a natural quotient map $q: DP \rightarrow P$ given by identifying all corresponding points. Let $DP^{\circ}$ be $DP$ with the vertices of the copies of $P$ removed. Note that $DP^{\circ}$ is a topological sphere with $n$ punctures. Let $\pi^{\circ}:X^{\circ}\to DP^{\circ}$ be a finite-sheeted
regular cover such that the involution on $DP^{\circ}$ interchanging the two copies of $P$ lifts to $X^{\circ}$.
Both $X^{\circ}$ and $DP^{\circ}$ have a flat metric induced by the metric on $P$. Let $X$ be the metric completion of $X^{\circ}$.  
As above, we assume that the cone angles are more than $2\pi$.
The covering map can be extended to a regular branched covering map $\pi: X \rightarrow DP$. Then $X$ is an unfolding of $P$, and the map $r = q\circ\pi$ is a folding map. Furthermore, $G$ is the group generated by the deck group of $\pi: X\to DP$ and a lift of the involution on $DP$. In particular, it follows that for any edge of the tessellation there is an element in $G$ that locally acts as a reflection in this edge. The quotient of $X$ by $G$ is precisely $P$. 
\end{definition}

\begin{remark}
In contrast to the theory of rational billiards,  we are considering all unfoldings, not just those that carry an induced translation structure.
The angle condition ensures that $X$ is nonpositively curved.  
We will use this construction to pass to simultaneous unfoldings of a pair of polygons $P_1,P_2$ to flat surfaces $X_1,X_2$ with 
the same underlying topology $S$.
\end{remark}

\begin{lemma}
\label{Liouville}
For an unfolding as above, if $\gamma: \mathbb{R}\to X$ is a unit speed nonsingular geodesic, then $r\circ\gamma: \mathbb{R}\to P$ is a billiard trajectory on $P$. In fact, this defines a bijection between $G$-orbits of nonsingular geodesics in $X$ and billiard trajectories on $P$.  
\end{lemma}

Viewing $r \colon X^\circ \to P^\circ$ as a Euclidean orbifold cover (and $P$ as a reflector orbifold), this lemma follows from standard orbifold generalizations of covering space theory (see \cite[Chapter 13]{Thurston:GT}). However, the proof is very constructive, so we sketch it here.

\begin{proof}[Proof of Lemma \ref{Liouville}]
Fix a nonsingular geodesic $\gamma$ on $X$. Since $r$ is a local isometry away from the edges of the tessellation by copies of $P$, 
it follows that $r\circ \gamma$ is a concatenation of geodesic segments between the edges of $P$. Moreover, we know that for every edge of the tessellation, there is an element of $G$ that acts locally as a reflection in that edge. Thus $r\circ\gamma$ has optical reflection at each point of intersection with an edge of $P$ (see Figure \ref{F:fold}), and $r\circ\gamma$ defines a billiard trajectory on $P$. 

Conversely, given a billiard trajectory, any lift to $X$ is a piecewise geodesic path and optical reflection is  precisely what is needed
to ensure that this piecewise geodesic is in fact an actual geodesic. 
Hence every nonsingular geodesic is a lift of a billiard trajectory, and by definition of unfolding, any two lifts differ by an element in $G$.  
\end{proof}

\begin{figure}[h]
\centering
\begin{tikzpicture}
\node (img) at (0,0) {\includegraphics[width=3in]{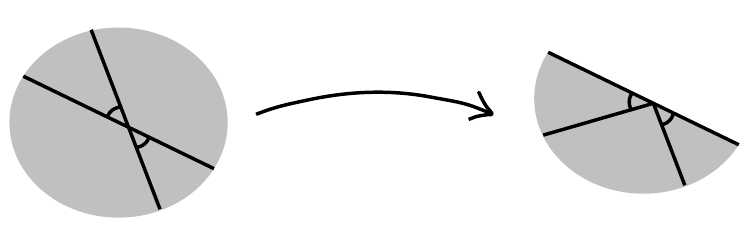}};
\node at (-60pt,-35pt) {$\gamma$};
\node at (-20pt,-15pt) {$\tilde{e}\in r^{-1}(e)$};
\draw[->] (87.2pt,-10.7pt) -- (87pt,-10pt);
\node at (0pt,20pt) {$r$};
\draw[->] (-67.5pt,-10.7pt) -- (-67.7pt,-10pt);

\node at (90pt,-30pt) {$r\circ\gamma$};
\node at (115pt,-8pt) {$e$};
\end{tikzpicture}
\caption{The image of a nonsingular geodesic $\gamma$ on $X$ under $r$ has optical reflection at points of intersection with an edge $e$ of $P$.}
\label{F:fold}
\end{figure}
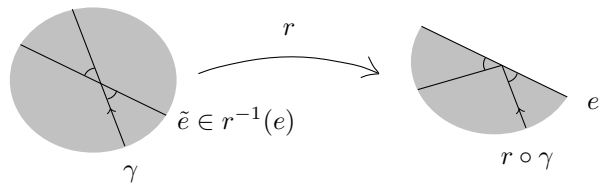

\begin{proof}[Proof of Bounce Theorem]
Suppose $\B(P_1) = \B(P_2)$. By counting observed labels, this implies that $P_1$ and $P_2$ have the same number of edges. This allows us to construct a homeomorphism $f:P_1 \rightarrow P_2$ respecting the cyclic labeling from the alphabet $\mathcal{A}$. We assume the restriction of $f$ to each edge of $P_1$ is affine onto the image edge of $P_2$.


The map $f \colon P_1 \to P_2$ lifts to a homeomorphism $Df \colon DP_1 \to DP_2$, and we denote the restriction to the complements of the vertices by $Df^\circ \colon DP^\circ_1 \to DP^\circ_2$, which is also a homeomorphism.  For each $i = 1,2$, we choose compact, nonpositively curved unfoldings $X_i$ and denote the folding map $r_i \colon X_i \to P_i$. We assume that the preimage of every vertex by $r_i$ is a cone point of angle greater than $2\pi$. The folding maps factor through branched coverings $X_i \to DP_i$ which restrict to (unbranched) covering maps $X^\circ_i \to DP^\circ_i$, and we can therefore identify the fundamental group $\pi_1X^\circ_i$ as a subgroup $\pi_1X^\circ_i < \pi_1DP^\circ_i$.  By passing to further finite-sheeted branched covers, we can assume that
\[ Df_* (\pi_1X^\circ_1) = \pi_1X^\circ_2\]
and hence $Df$ lifts to a homeomorphism $F \colon X_1 \to X_2$ fitting into a commutative diagram
\[ \xymatrix{
X_1 \ar[d] \ar[r]^{F} & X_2 \ar[d] \\
DP_1 \ar[r]^{Df} \ar[d] & DP_2 \ar[d]\\
P_1 \ar[r]^{f} & P_2. }
\]
Recall that for each $i = 1,2$, $X_i$ has a tessellation whose tiles are copies of $P_i$. The map $F$ sends tiles of $X_1$ to tiles of $X_2$ by construction.  Furthermore, the edges of the tessellation are labeled according to the labels on the original polygon, and $F$ is affine on the edges and preserves the labeling.\\

\noindent {\bf Claim.} The map $F \colon X_1 \to X_2$ is isotopic to an affine map $F' \colon X_1 \to X_2$ such that $F' = F$ on the edges of the tessellation.

\begin{proof} Suppose $\gamma_1 \colon \R \to X_1$ is any bi-infinite nonsingular geodesic.  By Lemma~\ref{Liouville}, the projection to $P_1$ is a billiard trajectory with bounce sequence ${\bf b}$.  This sequence is precisely the sequence of labels on edges of the tessellation crossed by $\gamma_1$.  Since $F$ preserves the labels on the edges of the tessellation, $F \circ \gamma_1$ crosses the edges creating the same sequence ${\bf b}$.  Since $\B(P_1) = \B(P_2)$, this is also the bounce sequence of a billiard trajectory in  $P_2$, which lifts to a bi-infinite nonsingular geodesic $\gamma_2 \colon \R \to X_2$ (again by Lemma~\ref{Liouville}) and produces the  same sequence of labels according to the edges it crosses.   By composing with an element of the group $G$ associated to $X_2$, we can assume that $\gamma_2$ crosses the same set of edges as $F \circ \gamma_1$, in the same direction.  In particular, after reparameterizing if necessary, we can find a homotopy between $F \circ \gamma_1$ and $\gamma_2$ in $X_2^\circ$ with uniformly bounded length traces (that is, the homotopy $H_t \colon \R \to X_2^\circ$, $t \in [0,1]$, from $F \circ \gamma_1$ to $\gamma_2$ has the property that the paths $t \mapsto H_t(x)$ have uniformly bounded length, independent of $x \in \R$; in fact, the length is bounded by the diameter of $P_2$ with its induced path metric).

Now let $\varphi_1$ be the given flat metric on $S = X_1$ and $\varphi_2$ the pullback of the flat metric on $X_2$ by the map $F$.  Observe that by construction, $\varphi_1$ and $\varphi_2$ have the same set of cone points.  The previous paragraph shows that any bi-infinite nonsingular $\tilde \varphi_1$--geodesic in the universal cover $\tilde \gamma_1 \colon \R \to \tilde S$ is uniformly close to a bi-infinite nonsingular $\tilde \varphi_2$--geodesic $\tilde \gamma_2 \colon \R \to \tilde S$.    Moreover, these are homotopic in the complement of the (common) cone point set $\widetilde \Sigma$.   Therefore the metrics have the same set of endpoints of  basic geodesics: 
$\G^*_{\tilde \varphi_1} = \G^*_{\tilde \varphi_2}$.  Furthermore, since the homotopy between a nonsingular $\tilde \varphi_1$--geodesic and a corresponding $\tilde \varphi_2$--geodesic occurs within $\tilde S \setminus \tilde \Sigma$, it follows that these geodesics define the same partition of $\tilde \Sigma$.

Next, recall that our identification of chains 
(Proposition~\ref{P:optimal Omega}) 
gives us a $\pi_1S$--invariant bijection $\tilde \Sigma \to \tilde \Sigma$, and in the proof of the Support Rigidity Theorem we 
extended this to a $\pi_1S$--equivariant affine map from $(\tilde S, \tilde \varphi_1) \to (\tilde S,\tilde \varphi_2)$, $\pi_1S$--equivariantly isotopic to the identity.  In fact, since any nonsingular $\tilde \varphi_1$--geodesic defines the same partition of $\tilde \Sigma$ as a corresponding nonsingular $\tilde \varphi_2$--geodesic and since any cone point can be uniquely determined by finitely many such partitions, 
it follows from Proposition~\ref{geodcomb} (Cone point partitions are well-defined)  
that this bijection $\tilde \Sigma \to \tilde \Sigma$ is the identity.

Therefore, we have determined that the identity on $\tilde S$ is $\pi_1S$--equivariantly isotopic, rel $\tilde \Sigma$, to a $\pi_1S$--equivariant affine map $\tilde \Phi \colon (\tilde S,\tilde \varphi_1) \to (\tilde S,\tilde \varphi_2)$.  Since each edge of the (lifted) tessellation of $\tilde S$ by copies of $P_1$ is both a $\tilde \varphi_1$--saddle connection as well as a $\tilde \varphi_2$--saddle connection, and the identity is already an affine map, it follows that $\tilde \Phi$ is the identity on these edges.  Equivalently, the induced map 
$\Phi \colon (S,\varphi_1) \to (S,\varphi_2)$ is an affine map, isotopic to the identity, which is the identity on the edges of the tessellation by copies of $P_1$.   The composition $F' = F \circ \Phi \colon X_1 \to X_2$ is an affine map, is isotopic to $F$, agreeing with 
$F$ on the edges of the tessellation.
\end{proof}

With the claim in hand, we finish the proof of the Bounce Theorem.
If $P_1$ is not a right-angled polygon, then there is a vertex with interior angle not in $\frac{\pi}2 \mathbb N$;  by an appropriate choice of unfolding $X_1$ of $P_1$, we can assume that the holonomy of $X_1$ has order greater than $2$ (specifically, we can force the holonomy around the cone point corresponding to the chosen vertex to have order greater than $2$).  By the Support Rigidity Theorem, the affine map $F' \colon X_1 \to X_2$ is a similarity (here we must allow a similarity, rather than isometry, since the two metrics have not been normalized to unit area).   In this case, the map $F' \colon X_1 \to X_2$ descends to a similarity $f' \colon P_1 \to P_2$, preserving the labeling, as required.

If $P_1$ is a right-angled polygon, then the affine map $F' \colon X_1 \to X_2$ still descends to a map $f' \colon P_1 \to P_2$, but we also note that it descends to an affine map $Df' \colon DP_1 \to DP_2$ on the doubles, 
sending each of the two copies of $P_1$ onto one of the copies of $P_2$.  By rotating if necessary, we may assume that $P_1$ has  horizontal and vertical sides.  Pick a horizontal side $e_1$ of $P_1$, and rotate $P_2$ so that the corresponding side $e_2$ of $P_2$ is also horizontal.  The derivative of $f' \colon P_1 \to P_2$ therefore has the form
$\left( \begin{smallmatrix} a&b\\ 0&c \end{smallmatrix}\right)$
for some $a,b,c \in \R$.  On the other hand, from $Df'$, we see that $f'$ extends to an affine map from the reflection of $P_1$ over $e_1$ to the reflection of $P_2$ over $e_2$, and hence its derivative commutes with the reflection in the $x$--axis.  An easy computation shows that  
$b = 0$, and hence $P_2$ is also right-angled. 
\end{proof}


\subsection{Generalized diagonals}\label{sec:diag}
Next, we prove  Corollary~\ref{C:gen diag}, which states that the (finite) bounce sequences of all generalized diagonals---billiard
trajectories starting and ending at a vertex---determine the shape of the polygon.  For the proof, we will need some additional setup.

Given a polygon $P$ with labels in $\mathcal A = \{1,\ldots,n \}$, add a "dummy'' label to form $\mathcal A_0 = \{0,1,\ldots,n\}$, and view $\B(P)$ as a subset of ${\mathcal A}^{\mathbb Z} \subset {\mathcal A}_0^{\mathbb Z}$.  A generalized diagonal $\tau \colon [a,b] \to P$ determines a finite bounce sequence, which we extend by $0$ to a bi-infinite sequence in ${\mathcal A}_0^{\mathbb Z}$, so for example the sequence $(1,2,3)$ extends to $(\ldots,0,0,1,2,3,0,0,\ldots)$.  We are interested in certain types of elements 
$(b_n) \in {\mathcal A}_0^{\mathbb Z}$, characterized by where they are zero and nonzero.  Specifically, let $I\subset \Z$ be such that $b_n \in \mathcal A$ for all $n \in I$ and $b_n = 0$ for $n \not \in I$, then we say that $(b_n)$ is {\em finite} if $I=[a,a']$, {\em forward infinite} if $I=[a,\infty)$, {\em backward infinite} if $I=(-\infty,a]$,
and {\em bi-infinite} if $I=\Z$, for appropriate $a,a'\in \Z$.
 Thus, the bounce sequence of a generalized diagonal is a finite sequence in ${\mathcal A}_0^{\mathbb Z}$ in this terminology.

We topologize $\mathcal A_0$ with the discrete topology and ${\mathcal A}_0^{\mathbb Z}$ with the product topology,  
making it  compact by Tychonoff's Theorem.  We then view the full bounce spectrum $\B(P)$ and the generalized diagonals
$\B_\Delta(P)$ as subspaces of ${\mathcal A}_0^{\mathbb Z}$.  We consider their closures
\[ \overline{\B(P)} \subset {\mathcal A}^{\mathbb Z} \mbox{ and } \overline{\B_\Delta(P)} \subset {\mathcal A}_0^{\mathbb Z}.\]
The closure on the left can be taken inside ${\mathcal A}_0^{\mathbb Z}$ instead, but since ${\mathcal A}^{\mathbb Z}$ is a closed subset, the closure is contained in this smaller set.
In fact, these closures can be constructed in a more geometric way, as we now describe.  

First take any unfolding $r \colon X \to P$ with symmetry group $G$ (so that $X/G = P$) and consider the natural tiling of $X$ by copies of $P$.  As in the proof of the Bounce Theorem, the nonsingular geodesics $\gamma$ on $X$, up to the action of $G$, correspond bijectively to billiard trajectories; recording the edge labels encountered by $\gamma$ on the copies of $P$ in $X$ produces 
its bounce sequence.

Next, we note that each basic geodesic also has an associated bounce sequence which records the labels from $P$ crossed by $\gamma$ in order.  Here we say that $\gamma$ {\em crosses a side $e$ of a copy of $P$} if either (a) it intersects the interior of $e$ transversely in a point, (b) it contains $e$ as a subsegment and switches from making angle $\pi$ on the left to angle $\pi$ on the right (or vice versa) at the endpoints of $e$, or (c) it contains an endpoint of $e$ and makes supplementary angles with that side.  Condition (c) is equivalent to saying that $\gamma$ contains the endpoint $v$ of $e$, and that $e$ lies on the side of $\gamma$ which makes angle $\pi$ at $v$.   
When $\gamma(0)$ is at a vertex, there are finitely many ways to decide which of the edge crossings (in the sense of (b) and (c)) is recorded in the index-0 position. This means that there are a uniformly bounded number of such sequences for a given $\gamma$ (with the bound depending on the angles of the polygon), and they all differ by shifts.

Now consider any sequence of bounce sequences ${\bf b}_n$ converging in ${\mathcal A}^{\mathbb Z}$ to a sequence ${\bf b} \in {\mathcal A}^{\mathbb Z}$. For each ${\bf b}_n$ choose a bi-infinite nonsingular geodesic $\gamma_n: \mathbb{R}\to X$ having ${\bf b}_n$ as its associated bounce sequence and such that $\gamma_n(0)$ lies on an edge labeled by the zeroth term of ${\bf b}_n$.  After passing to a subsequence if necessary, we can assume that $\gamma_n$ converges to a basic geodesic $\gamma$ in $X$ (or more precisely, the projection to $X$ of a basic geodesic in the universal cover).  By the construction of bounce sequences for basic geodesics, ${\bf b}$ is one of the bounce sequences for $\gamma$ (which is unique unless $\gamma(0)$ is at a vertex). 
Since every basic geodesic is a limit of nonsingular geodesics, it follows that $\overline{\B(P)}$ consists precisely of the set of bounce sequences of basic geodesics.

One can carry out a similar construction for limits of generalized diagonals, recalling that a generalized diagonal on $P$ is the image of a saddle connection on $X$.  Suppose we have a sequence  of bounce sequences ${\bf b}_n\in \B_\Delta(P)$ limiting to ${\bf b} \in {\mathcal A}_0^{\mathbb Z}$.  For each ${\bf b}_n$ consider an associated saddle connection $\gamma_n \colon [a,a'] \to X$, with $a \leq 0 \leq a'$, such that $\gamma(0)$ lies on an edge corresponding to the zeroth term of ${\bf b}_n$. Pass to a subsequence so that $\gamma_n$ converges to a geodesic $\gamma$.
If the lengths of $\gamma_n$ are uniformly bounded, then $\gamma_n$ must eventually be constant, and so ${\bf b}_n = {\bf b}$ for $n$ sufficiently large.  Since this case is trivial, we assume that the length of $\gamma_n$ tends to infinity.
In this case, $\gamma$ is either a bi-infinite basic geodesic or a basic geodesic ray (i.e., a subray of a basic geodesic).  As above, we can associate (a finite set of) bounce sequences to $\gamma$ and ${\bf b}$ is one of these.
Thus, all sequences ${\bf b}$ in $\overline{\B_\Delta(P)}$ are bounce sequences of saddle connections, basic geodesics, or basic geodesic rays (i.e., subrays of basic geodesics emanating from a cone point).

We will need the following lemma that allows us to detect, via $\B_\Delta(P)$, the bounce sequences of basic geodesics that meet cone points.  The only real difficulty in the proof comes from the awkward definition of the bounce sequence(s) for a basic geodesic described above, and the fact that periodic sequences arise as bounce sequences of both singular and nonsingular basic geodesics.

\begin{lemma} \label{L:find singular} A sequence ${\bf b} = (b_n) \in \overline{\B(P)}$ is the bounce sequence of a singular basic geodesic if and only there exists a forward infinite (respectively, backward infinite) sequence ${\bf b}' = (b_n') \in \overline{\B_\Delta(P)}$ and integer $N$ such that, $b_n = b_n'$ for all $n  > N$ (respectively, $n < N$).
\end{lemma}
\begin{proof} First choose an unfolding $r \colon X \to P$, and let $p \colon \widetilde X \to X$ denote the universal covering.  Consider the tessellation of $\widetilde X$ by copies of $P$ and let $\widetilde G$ be the symmetry group, so that $\widetilde X/\widetilde G = P$.  

Given any basic geodesic $\gamma \colon I \to \widetilde X$ defined on an interval $I \subset \mathbb R$ and $\epsilon > 0$, we can perturb $\gamma$ to a path $\gamma_\epsilon \colon I \to  \widetilde X$, so that $d(\gamma_\epsilon(t),\gamma(t)) < \epsilon$ and $\gamma_\epsilon$ transversely crosses the interiors of the edges of the tessellation of $\widetilde X$ by copies of $P$ according to the bounce sequence ${\bf b}$ of $\gamma$, and is disjoint from the cone points (if $\gamma$ is nonsingular, $\gamma = \gamma_\epsilon$).  A path $\gamma_\epsilon$ with this property is unique up to homotopy in the complement of the cone points through paths with this property. See Figure~\ref{F:perturbing basic}.

\begin{center}
\begin{figure}[htb]
\begin{tikzpicture} [scale=.8]
\node at (-4,2) {$\quad$};
\draw[line width=1pt,->] (.5,1.5) -- (.75,2.25);
\draw[line width=1pt,->] (.5,1.5) -- (1.5,4.5);
\draw[line width=1pt,->] (1.5,4.5) -- (2.5,7.5);
\node at (2.5,6.6) {$\gamma$};
\begin{scope}[line width = 2pt]
\draw (1,3) -- (2,6);
\draw (1,3) -- (3,3);
\draw (3,2) -- (1,3) -- (3,5);
\draw (0,5) -- (2,6) -- (-1,7);
\draw  (2,6) -- (0,8);
\end{scope}
\draw  (.9,2.7) arc [radius=.4, start angle =-120, end angle = 80];
\node at (1.1,2.5) [below,right] {$\pi$};
\draw  (2.1,6.3) arc [radius=.4, start angle =80, end angle = 240];
\node at (2,6.5) {$\pi$};
\end{tikzpicture}
\hspace{3cm}
\begin{tikzpicture} [scale=.8]
\draw[line width=1pt,->] (.5,1.5) -- (.75,2.25);
\draw[line width=1pt,->] (.75,2.25) .. controls (1.6,3) and (1.2,4) .. (1.3,4.5);
\draw[line width=1pt] (1.3,4.5) .. controls (1.6,6) and (2.2,6.6) ..  (2.4,7.2);
\draw[line width=1pt,->] (2.4,7.2) -- (2.5,7.5);
\node at (2.5,6.6) {$\gamma_\epsilon$};
\begin{scope}[line width = 2pt]
\draw (1,3) -- (2,6);
\draw (1,3) -- (3,3);
\draw (3,2) -- (1,3) -- (3,5);
\draw (0,5) -- (2,6) -- (-1,7);
\draw  (2,6) -- (0,8);
\end{scope}
\end{tikzpicture}
\caption{Replacing a singular basic geodesic $\gamma$ with a perturbation $\gamma_\epsilon$. The sides crossed by $\gamma$ are easily read off from $\gamma_\epsilon$.} \label{F:perturbing basic}
\end{figure}
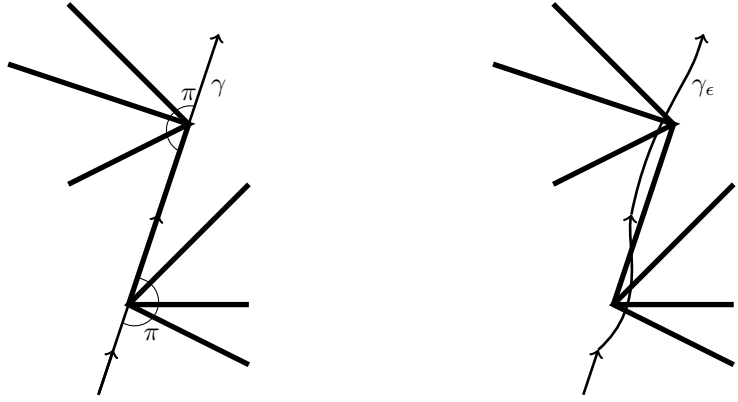
\end{center}

Next suppose $\gamma$ with bounce sequence ${\bf b} = (b_n)$ and $\gamma'$  with bounce sequence ${\bf b'} = (b_n')$ are two basic geodesics and $b_n = b_n'$ for all $n > N$, for some integer $N$.
Further suppose that $\gamma$ is nonsingular and let $\gamma'_\epsilon$ be as above.  As in the proof of the Bounce Theorem (see the proof of the claim), after applying an element of $\widetilde G$, we can assume that subrays of $\gamma$ and $\gamma'_\epsilon$ cross the same $P$-tiles in the same order.  In particular, $\gamma$ and $\gamma'_\epsilon$ are asymptotic in the forward direction.  Since $\gamma'$ and $\gamma'_\epsilon$ are $\epsilon$ apart, $\gamma$ and $\gamma'$ are forward asymptotic.  Then either $\gamma$ and $\gamma'$ are cone point asymptotic, or there are subrays that bound a half-strip.  The former cannot happen because $\gamma$ contains no cone points.  Therefore, $\gamma$ and $\gamma'$ contain subrays that bound a half-strip.  In particular, their bounce sequences are necessarily periodic in the forward direction by Theorem~\ref{T:flat strips}.  We can make a similar argument if $b_n = b_n'$ for all $n < N$.

Now we let $\gamma$ be any basic geodesic with bounce sequence ${\bf b} = (b_n) \in \overline{\B(P)}$.
If $\gamma $ is singular, then let $\gamma'$ be a basic geodesic subray of $\gamma$ starting at a cone point.  Its bounce sequence ${\bf b'} = (b_n') \in \overline{\B_\Delta(P)}$ has $b_n = b_n'$ for all $n$ sufficiently large.
Next, suppose that there is a basic geodesic ray $\gamma'$ emanating from a cone point with bounce sequence ${\bf b'} = (b_n') \in \overline{\B_\Delta(P)}$ and $b_n = b_n'$ for all $n$ sufficiently large (the argument for negative infinite sequences is similar).  We suppose that $\gamma$ is nonsingular, and prove that there is a singular geodesic with the same bounce sequence as $\gamma$.  Extend $\gamma'$ to a biinfinite singular basic geodesic $\gamma''$ with bounce sequence ${\bf b''} = (b_n'')$.  Note that $b_n'' = b_n' = b_n$ for all sufficiently large $n$.  By the previous paragraph, since $\gamma$ is nonsingular and $\gamma''$ is singular, it must be the case that there are subrays of each that bound a half-strip.  Therefore, ${\bf b}$ (and hence ${\bf b'}$) are periodic in the forward direction by Theorem~\ref{T:flat strips}.  This implies that $\gamma$ is entirely contained in a flat strip (because it is nonsingular). The boundary geodesic $\gamma'''$ of the flat strip has the same bounce sequence as $\gamma$, and is singular. This completes the proof of the lemma.
\end{proof}

We are now ready for the proof of Corollary~\ref{C:gen diag}.

\bigskip

\noindent {\bf Corollary~\ref{C:gen diag}.} {\em If two simply connected Euclidean polygons $P_1,P_2$ have 
$\B_\Delta(P_1)=\B_\Delta(P_2)$, then either $P_1,P_2$ are right-angled and affinely equivalent, or they are 
similar polygons.}
\smallskip

\begin{proof}  By the Bounce Theorem, it is enough to show that $\B_\Delta(P)$ uniquely determines $\B(P)$.  
First note that any biinfinite limit of saddle connections is easily seen to be a limit of nonsingular geodesics.  Consequently, we have
\begin{equation} \label{E:biinf saddle limits}
\overline{\B_\Delta(P)} \cap {\mathcal A}^{\mathbb Z} = \overline{\B(P)}.
\end{equation}
Each ${\bf  b} \in \overline{\B(P)}$ is the bounce sequence of a basic geodesic, so it suffices do determine which ${\bf b} \in \overline{\B(P)}$ are bounce sequences of nonsingular geodesics -- this is exactly $\B(P)$.

For this, we apply Lemma~\ref{L:find singular} and note that we can decide if ${\bf b} \in \overline{\B(P)}$ is the bounce sequence of a singular geodesic from the set $\B_\Delta(P)$.  Any ${\bf b}$ that is a bounce sequence for both a singular and nonsingular geodesic is a periodic bounce sequence, as shown in the proof of Lemma~\ref{L:find singular}.   By a result of Gal'perin--Kr\"uger--Troubetzkoy \cite{GKT}, any periodic bounce sequence is the bounce sequence of a periodic billiard trajectory, and hence of a nonsingular geodesic.  Therefore, $\B(P)$ is given by removing from $\overline{\B(P)}$ all bounce sequences of singular geodesics (determined by $\B_\Delta(P)$), then adding back in the periodic bounce sequences that were removed in the process.  
\end{proof}


\subsection{The case of rational billiards} 
We note that we can also deduce the following stronger version of the Bounce Theorem (first proved by Bobok-Troubetzkoy \cite{BT3} using different methods) under the assumption that all angles of $P_1$ are rational multiples of $\pi$: \emph{If two tables $P_1$ and $P_2$, with sides cyclically labeled by $\mathcal A$ have the same periodic bounce spectra, and $P_1$ is rational, then $P_1$ and $P_2$ are either similar or they are
right-angled and affinely equivalent.} Here the \emph{periodic} bounce spectrum refers to the subset $\B_{p}(P)\subset\B(P)$ 
consisting of all periodic bounce sequences. In fact, a bounce sequence is periodic if and only if the corresponding billiard trajectory is periodic, as shown in \cite{GKT}. 

To see how we can prove this result using the methods presented here, we first explain why $P_1$ being rational and having the same periodic bounce spectrum as $P_2$ implies that $P_2$ is also rational. If all angles of $P_1$ are rational multiples of $\pi$, then we can unfold  to a {\em translation surface} $S = X_1$ with flat metric $\varphi_1$ (that is, $\varphi_1$ has trivial holonomy).  We also choose a nonpositively curved unfolding $X_2$ of $P_2$ with the same underlying topological surface $S$ and let $\varphi_2$ be the pullback of the metric on $X_2$, as in the proof of the Bounce Theorem.  Since $X_1$ is a translation surface, a result of Vorobets \cite{vorobets} shows that tangent vectors to core geodesics of cylinders are dense in $T^1X_1$ (see also \cite{MasurClosed}).  Consequently, the endpoints of the lifts of cylinder curves are dense in the support of the Liouville current $L_{\varphi_{1}}$.  Since $B_p(P_1)$ determines the cylinder curves in $X_1$, it also determines $\supp(L_{\varphi_{1}})$, and consequently all its chains.  As pointed out in \cite[Section 6]{BL}, the chains also determine the angles of the cone points, and hence we know all the cone points of $\varphi_1$ as well as their cone angles.  Although we do not know, a priori, that endpoints of lifts of cylinder curves for $\varphi_2$ are dense in $\supp(L_{\varphi_2})$, we do know that the closure of these endpoints is contained in $\supp(L_{\varphi_2})$.  Since $\B_p(P_1) = \B_p(P_2)$, it follows that the two metrics have the same set of cylinder curves, and hence $\supp(L_{\varphi_1}) \subset \supp(L_{\varphi_2})$.

From this one can show that a generic\footnote{Chains are defined in \cite{BL} in terms of pairs of geodesics in $\supp(L_{\varphi})$ which are asymptotic to each other in one direction, with no other geodesics from $\supp(L_{\varphi})$ between them.  It is conceivable that two such geodesics in the set $\supp(L_{\varphi_1}) \subset \supp(L_{\varphi_2})$ may nonetheless have a geodesic in $\supp(L_{\varphi_2})$ between them.  This can happen only countably many times (i.e., {\em non-generically}) as a consequence of Theorem~\ref{T:flat strips}, and so we may ignore these.} chain for the $\varphi_1$-metric is also a chain for the $\varphi_2$-metric, giving a $\pi_1S$--equivariant  {\em injection} from the cone points of $\tilde \varphi_1$  to those of $\tilde \varphi_2$ in $\tilde S$.
Therefore, there is a cone-angle-preserving injection from the cone points of $\varphi_1$ to those of $\varphi_2$.  By Gauss-Bonnet, this accounts for all cone points, and thus the injection is actually a bijection.
Therefore the angles of $P_2$ must be the same as those of $P_1$, and in particular all angles of $P_2$ are rational multiples of $\pi$ as well.

Once we know both tables are rational, applying \cite{vorobets} again, we can deduce that $\supp(L_{\varphi_1}) = \supp(L_{\varphi_2})$.  Now we can follow the proof of the Bounce Theorem above verbatim to complete the argument.

\subsection{Cutting sequences} \label{S:cutting seq}
We end by explaining how the application of the Support Rigidity Theorem to the Bounce Theorem can be extended to other settings, in particular to the context of \emph{cutting sequences}. 
Recall that a cutting sequence is a symbolic coding of linear trajectories on translation surfaces, constructed by gluing polygons (see, for example, \cite{SmillieUlcigrai, DavPasUlc}). More generally, suppose we have a Euclidean polygon $P$ and we glue pairs of edges by 
isometries, obtaining an orientable surface $S$ with cone metric $\varphi$.  Given a cyclic labeling of the edges of $P$, the side pairing defines an equivalence relation on the labels.
Consider a bi-infinite geodesic $\gamma$ on $S$ that avoids the vertices of $P$. Recording the labels crossed by $\gamma$  
determines a bi-infinite symbolic coding of the geodesic, which we call its cutting sequence.  We note that each edge actually has two labels (since the edges are paired) and we record the one at which we enter, not the one from which we exit; see Figure~\ref{octagons}.
We note that our definition of a cutting sequence is different from the usual definition since this double labeling essentially records the \emph{direction} in which a geodesic crosses an edge instead of just the edge it crosses.

Let $P_1, P_2$ be two cyclically labeled, oriented polygons with side pairings given by isometries  as above.
We say that $P_1$ and $P_2$ are \emph{combinatorially equivalent} if they have the same set of labels, and if the side pairings induce the same equivalence relations; see Figure \ref{octagons}.
We claim that if we have combinatorially equivalent polygons with the same set of cutting sequences, then the polygons can only differ by an affine deformation. To see this, note that combinatorial equivalence implies that there is a homeomorphism $P_1 \to P_2$ respecting the labelings and pairings, and hence a homeomorphism of the resulting glued surfaces $f \colon S_1 \to S_2$.
If the corresponding two metrics $\varphi_1,\varphi_2$ have all cone angles greater than $2\pi$, then since the cutting sequences determine the bi-infinite geodesics, the Support Rigidity Theorem applies as in the proof of the Bounce Theorem above to produce an affine map from $(S_1,\varphi_1)$ to $(S_2,\varphi_2)$ that restricts to an affine homeomorphism $P_1 \to P_2$.  If not, we can pass to finite-sheeted covers of $S_1$ and $S_2$ that are non-positively curved, and run the same argument there.

\begin{theorem}
Let $P_1, P_2$ be two labeled, oriented, simply connected polygons with side pairings given by isometries. If $P_1$ and $P_2$ are combinatorially equivalent and have the same set of cutting sequences, then they are affinely equivalent.  \qed
\end{theorem}

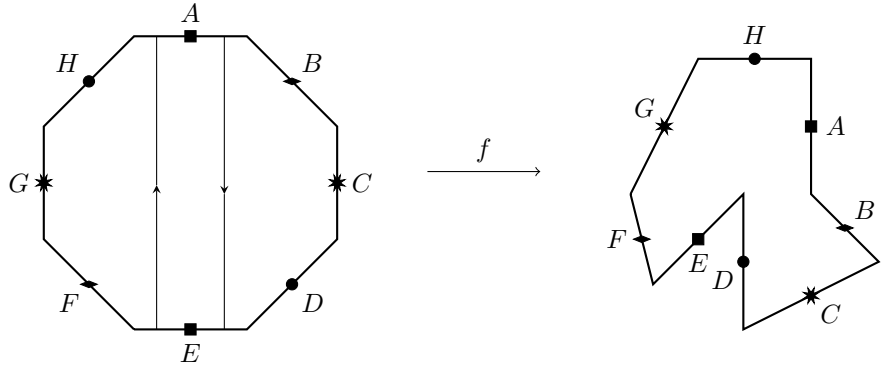
\begin{figure}[h]
\centering 
\begin{tikzpicture} [scale = .3]
\draw [thick] (4,0) -- (9,0) -- (13,4) -- (13,9) -- (9,13) -- (4,13) -- (0,9) -- (0,4) -- (4,0); 

\node [draw,fill=black,regular polygon, regular polygon sides=4, inner sep=1.5pt] at (6.5,0) {};
\node [draw,fill=black,regular polygon, regular polygon sides=4, inner sep=1.5pt] at (6.5,13) {};

\node [draw,fill=black,star, star points=8, star point ratio=.4, inner sep=2.2pt] at (0,6.5) {};
\node [draw,fill=black,star, star points=8, star point ratio=.4, inner sep=2.2pt] at (13,6.5) {};
\node [draw,fill=black,star, star points=2, star point ratio=.4, inner sep=2.2pt] at (2,2)  {};
\node [draw,fill=black,star, star points=2, star point ratio=.4, inner sep=2.2pt] at (11,11)  {};
\node [draw,fill=black,circle, inner sep=1.5pt]  at (11,2)  {};
\node [draw,fill=black,circle, inner sep=1.5pt]  at (2,11)  {};
\node at (6.5,13) [above=2pt] {$A$};
\node at (11,11) [above right] {$B$};
\node at (13,6.5) [right=2pt] {$C$};
\node at (11,2) [below right] {$D$};
\node at (6.5,0) [below=2pt] {$E$};
\node at (2,2) [below left] {$F$};
\node at (0,6.5) [left=2pt] {$G$};
\node at (2,11) [above left] {$H$};

\draw [thick] (26,6) -- (29,12) -- (34,12) -- (34,6) -- (37,3) -- (31,0) -- (31,6) -- (27,2) -- (26,6); 
\node [draw,fill=black,regular polygon, regular polygon sides=4, inner sep=1.5pt] at (29,4) {};
\node [draw,fill=black,regular polygon, regular polygon sides=4, inner sep=1.5pt] at (34,9) {};

\node [draw,fill=black,star, star points=2, star point ratio=.4, inner sep=2.2pt] at (35.5,4.5) {};
\node [draw,fill=black,star, star points=2, star point ratio=.4, inner sep=2.2pt] at  (26.5,4)  {};

\node [draw,fill=black,star, star points=8, star point ratio=.4, inner sep=2.2pt] at (34,1.5) {};
\node [draw,fill=black,star, star points=8, star point ratio=.4, inner sep=2.2pt] at (27.5,9) {};

\node [draw,fill=black,circle, inner sep=1.5pt]  at (31,3)  {};
\node [draw,fill=black,circle, inner sep=1.5pt]  at (31.5,12)  {};

\draw[->] (17,7) -- node [above] {$f$} (22,7);
\node at (31.5,12) [above=2pt] {$H$};
\node at (34,9) [right=2pt] {$A$};
\node at (35.5,4.5) [above right] {$B$};
\node at (34,1.5) [below right] {$C$};
\node at (31,3) [below left] {$D$};
\node at (29,4) [below=2pt] {$E$};
\node at (26.5,4) [left=2pt] {$F$};
\node at (27.5,9) [above left] {$G$};
\draw[-stealth] (5,0) -- (5,6.4);
\draw (5,6.4) -- (5,13);
\draw[-stealth] (8,13) -- (8,6);
\draw (8,6) -- (8,0);
\end{tikzpicture}
\caption{Two combinatorially equivalent octagons and the bijection $f$. There are two periodic, vertical trajectories shown in the left octagon:   one has cutting sequence $\{\ldots, A,A,A,\ldots \}$ and the other has  cutting sequence $\{ \ldots, E,E,E, \ldots \}$. }\label{octagons}
\end{figure}

  \bibliographystyle{alpha}
  \bibliography{main}

\end{document}